\documentclass[a4paper,12pt]{article}
\usepackage{pstricks}
\usepackage{pstricks-add}
\usepackage{amsmath, amsfonts, amssymb, amstext, amsthm} 
\usepackage{graphicx, subfigure} 
\usepackage[affil-it]{authblk} 
\usepackage{tikz}

\def\R{\mathbb{R}}
\def\N{\mathbb{N}}

\def\e{\varepsilon}

\def\a{\alpha}
\def\b{\beta}
\def\la{\lambda}
\def\d{\delta}

\def\up{\upsilon}
\def\pl{\phi_\la}
\def\pll{\phi_{\la\la}}
\def\nub{\overline{\nu}}
\def\mub{\overline{\mu}}
\def\ft{\widetilde{f'}(0)}
\def\fp{f'(0)}
\def\Pt{\widetilde{P}}
\def\vp{\varphi}
\def\Gt2{\widetilde{\Gamma_2}}
\def\su{\underline{u}}
\def\sv{\underline{v}}
\def\us{\overline{u}}
\def\vs{\overline{v}}

\def\las{\overline{\lambda}}
\def\phis{\underline{\phi}}

\def\Psit{\tilde{\Psi}}

\def\vps{\underline{\varphi}}

\def\ut{\tilde{u}}
\def\vt{\tilde{v}}
\def\ct{\tilde{c}}
\def\Cu{C_{unif}}
\def\svp{\underline{\varphi}}
\def\vps{\overline{\varphi}}

\renewcommand{\Re}[1]{{\cal R}e  \left ( #1 \right )}

\def\lp{\left(}
\def\rp{\right)}
\def\lb{\left|}
\def\rb{\right|}
\def\lV{\left\Vert}
\def\rV{\right\Vert}

\def\ML{\mathcal{L}}
\def\MV{\mathcal{V}}

\newtheorem{theorem}{Theorem}[section]
\newtheorem{prop}{Proposition}[section]
\newtheorem{lem}{Lemma}[section]
\newtheorem{cor}{Corollary}

\theoremstyle{definition}

  \author{Antoine Pauthier%
  \thanks{e-mail: \texttt{antoine.pauthier@math.univ-toulouse.fr}}}
\affil{Institut de Math\'ematiques de Toulouse ; UMR5219 \\ Universit\'e de Toulouse ; CNRS \\ UPS IMT, F-31062 Toulouse Cedex 9, France}

\title{The influence of nonlocal
	exchange terms on Fisher-KPP propagation driven by a line of fast diffusion}

\begin{document}
 
\maketitle
\tableofcontents

\begin{abstract}
A new model to describe biological invasion influenced by a line with fast diffusion has been introduced
by H. Berestycki, J.-M. Roquejoffre and L. Rossi in 2012.The purpose of this article is to present a related model where the line
of fast diffusion has a nontrivial range of influence, i.e. the exchanges between the line and the surrounding 
space has a nontrivial support.  We show the existence of a  spreading velocity depending on the diffusion on the line. 
Two intermediate model are also discussed.
Then, we try to understand
the influence of different exchange terms on this spreading speed. We show that various behaviour may happen, depending 
on the considered exchange distributions. 
\end{abstract}

\section{Introduction}
\subsection{Models}

The purpose of this study is a continuation of ~\cite{BRR1} in which was introduced, by H. Berestycki, J.-M. Roquejoffre
and L. Rossi, a new model to 
describe biological invasions in the plane when a strong diffusion takes place on a line, given by (\ref{BRReq2}).

\begin{equation}
\label{BRReq2}
\begin{cases}
\partial_t u-D \partial_{xx} u= \nub v(t,x,0)-\mub u & x\in\R,\ t>0\\
\partial_t v-d\Delta v=f(v) & (x,y)\in\R\times\R^*,\ t>0\\
v(t,x,0^+)=v(t,x,0^-), & x\in\R,\ t>0 \\
-d\left\{ \partial_y v(t,x,0^+)-\partial_y v(t,x,0^-) \right\}=\mub u(t,x)-\nub v(t,x,0) & x\in\R,\ t>0.
\end{cases}
\end{equation}
A two-dimensional environment (the plane $\R^2$) includes a line (the line $\{(x,0),\quad x\in \R \}$) in which 
fast diffusion takes place while reproduction and usual diffusion only occur outside the line. For the sake of simplicity, we will refer 
to the plane as ``the field`` and the line as ``the road``, as a reference to the biological situations. 
The density of the population is designated by $v=v(t,x,y)$ in the field, 
and $u=u(t,x)$ on the road. Exchanges of densities take place between the field and 
the road: a fraction $\nu$ of individuals from the field at the road (i.e. $v(x,0,t)$) joins the road, while a fraction 
$\mu$ of the population on the road joins the field. The diffusion coefficient in the field is $d$, on the road $D$. Of course, 
the aim is to study the case $D>d$.
The nonlinearity $f$ is of Fisher-KPP type, i.e. strictly concave with $f(0)=f(1)=0$.
Considering a nonnegative, compactly supported initial datum
$(u_0,v_0)\neq(0,0)$, the main result of \cite{BRR1} was the existence of an  asymptotic speed of spreading $c^*$ in
the direction of the road. 
They also explained the dependence of $c^*$ on $D,$ the coefficient of diffusion on the road.
In their model, the line separates the plane in two half-planes which do not interact with each other, but only with the line. Moreover,
 interactions between a half-plane and the line occur only with the limit-condition in (\ref{BRReq2}). That is why, in ~\cite{BRR1}, 
the authors consider only a half-plane as the field.

New results on (\ref{BRReq2}) have been recently proved. Further effects like a drift or 
a killing term on the road have been investigated in \cite{BRR2}. The case of a fractional diffusion on the road was studied 
and explained by the three authors and A.-C. Coulon in \cite{BRRC} and \cite{these_AC}. Models with an ignition-type nonlinearity
are also studied by L. Dietrich in \cite{Dietrich1} and \cite{Dietrich2}.

Our aim is to understand what happens when local interactions are replaced by integral-type interactions: exchanges of populations may happen 
between the road and a point of the field, not necessarily at the road. The density of individuals who jump from a 
point of the field to the road is represented by $y\mapsto \nu(y)$, from the road to a point of the field by 
$y \mapsto \mu(y)$. This is a more general model than the previous one, but interactions still only occur in one dimension, the y-axis.
We are led to the following system: 
\begin{equation}
 \label{RPeq}
\begin{cases}
 \partial_t u-D \partial_{xx} u = -\mub u+\int \nu(y)v(t,x,y)dy & x \in \R,\ t>0 \\
 \partial_t v-d\Delta v = f(v) +\mu(y)u(t,x)-\nu(y)v(t,x,y) & (x,y)\in \R^2,\ t>0,
\end{cases}
\end{equation}
where $\mub = \int \mu(y)dy$, the parameters $d$ and $D$ are supposed constant positive, $\mu$ and $\nu$ are
supposed nonnegative, and $f$ is a reaction term of KPP type.
Using the notation $\nub=\int\nu,$ we can generalise this to exchanges given by boundary
conditions, with $\mu=\mub\d_0$ and $\nu=\nub\d_0.$
Hence, in the same vein as (\ref{RPeq}), it is natural to consider the following semi-limit model
\begin{equation}
\label{RPSL}
\begin{cases}
\partial_t u-D \partial_{xx} u=-\mub u +\int\nu(y)v(t,x,y)dy & x\in\R,\ t>0\\
\partial_t v-d\Delta v=f(v) -\nu(y)v(t,x,y) & (x,y)\in\R\times\R^*,\ t>0\\
v(t,x,0^+)=v(t,x,0^-), & x\in\R,\ t>0 \\
-d\left\{ \partial_y v(t,x,0^+)-\partial_y v(t,x,0^-) \right\}=\mub u(t,x) & x\in\R,\ t>0
\end{cases}
\end{equation}
where interactions from the road to the field are local whereas interactions from the field to the road are still nonlocal.
We also introduce te symmetrised semi-limit model, where nonlocal interactions are only from the road to the field.
\begin{equation}
\label{RPSL2}
\begin{cases}
\partial_t u-D \partial_{xx} u=-\mub u +\nub v(t,x,0) x\in\R,\ t>0\\
\partial_t v-d\Delta v=f(v) +\mu(y)u(t,x) & (x,y)\in\R\times\R^*,\ t>0\\
v(t,x,0^+)=v(t,x,0^-), & x\in\R,\ t>0 \\
-d\left\{ \partial_y v(t,x,0^+)-\partial_y v(t,x,0^-) \right\}=-\nub v(t,x,0) & x\in\R,\ t>0.
\end{cases}
\end{equation}
All these models are connected with each other, setting the scaling
$$
\nu_\e(y)=  \frac{1}{\e}\nu\lp\frac{y}{\e}\rp,\ \mu_\e(y)=  \frac{1}{\e}\mu\lp\frac{y}{\e}\rp.
$$
With this scaling, exchanges functions tends to Dirac functions, and integral exchanges tends formally to boundary conditions.
For example, the limit $\e\to0$ in (\ref{RPeq}) leads to the dynamics of (\ref{BRReq2}). This result will be investigating in \cite{Pauthier2}.
A similar study would yield to the same kind of convergence of systems (\ref{RPSL}) or (\ref{RPSL2}) to (\ref{BRReq2}).

Reaction-diffusion equations of the type 
$$
\partial_t u-d\Delta u=f(u)
$$
have been introduced in the celebrated articles of Fisher ~\cite{fisher} and Kolmogorov, Petrovsky and Piskounov ~\cite{KPP} in 1937.
The initial motivation came from population genetics. The reaction term are that of a logistic law, whose
archetype is $f(u)=u(1-u)$ for the simplest example. In their works in one dimension, 
Kolmogorov, Petrovsky and Piskounov revealed the existence of propagation waves, together with an asymptotic speed of spreading
of the dominating gene, given by $2\sqrt{df'(0)}$. The existence of an asymptotic speed of spreading was generalised in $\R^n$ 
by D. G. Aronson and H. F. Weinberger in ~\cite{AW} (1978). Since these pioneering works, front propagation in 
reaction-diffusion equations have been widely studied. Let us cite, for instance, the works of Freidlin and G\"artner \cite{FG}
for an extension to periodic media, or \cite{W2002}, \cite{BHN1} and \cite{BHN2} for more general domains.

\subsection{Assumptions}
We always assume that $u_0$ and $v_0$ are nonnegative, bounded and uniformly continuous, with $(u_0,v_0)\not\equiv(0,0)$. Our assumptions on the reaction term
are of KPP-type: 
$$
f\in C^1([0,1]), \ f(0)=f(1)=0, \ \forall s\in (0,1),\ 0<f(s)\leq f'(0)s.
$$
We extend it to quadratic negative function outside $[0,1].$ Our assumptions on the exchange terms will differ depending on the sections.
For the parts concerning the robustness of the results of \cite{BRR1}, that is Proposition \ref{liouville} and Theorem \ref{spreadingthm}, they are the following: 
\begin{itemize}
 \item $\mu$ is supposed to be nonnegative, continuous, and decreasing faster than an exponential function: $\exists M>0,\ a>0$ such that
$\forall y\in \R,\ \mu(y)\leq M\exp(-a|y|).$
 \item $\nu$ is supposed to be nonnegative, continuous and twice integrable, both in $+\infty$ and $-\infty$, id est 
\begin{equation}
 \label{nucond}
 \int_0^{+\infty}\int_x^{+\infty}\nu(y)dydx<+\infty, \  \int_{-\infty}^{0}\int_{-\infty}^x\nu(y)dydx<+\infty
\end{equation}
 \item We suppose $\mu,\nu \not\equiv 0$, $\nu(0)>0,$ and that both $\nu$ and $\mu$ tend to $0$ as $|y|$ tends to $+\infty.$
\end{itemize}

For the parts dealing with variations on the spreading speed, we suppose that $\nu$ and $\mu$ are either nonnegative, continuous, compactly supported even functions, either given by a Dirac measure, either
 the sum of a Dirac measure and a nonnegative, continuous, compactly supported even function.

\subsection{Results of the paper}
\subsubsection{Persistence of the results of \cite{BRR1}}
We start with the results that are similar in flavour to those of \cite{BRR1} concerning the system (\ref{BRReq2}) and showing the 
robustness of the threshold $D=2d$ which was brought out in the paper. The first one concerns the stationary solutions of (\ref{RPeq}) and
the convergence of the solutions to this equilibrium.
\begin{prop}
 \label{liouville}
 under the assumptions on $f$, $\nu$, and $\mu$, then: 
 \begin{enumerate}
  \item problem (\ref{RPeq}) admits a unique positive bounded stationary solution $(U_s,V_s)$, which is x-independent ;
  \item for all nonnegative and uniformly continuous initial condition $(u_0,v_0)$, the solution $(u,v)$ of (\ref{RPeq}) starting
  from $(u_0,v_0)$satisfies
  $\displaystyle
  \left(u(t,x),v(t,x,y)\right)\underset{t\to\infty}{\longrightarrow}(U_s,V_s)
  $
  locally uniformly in $(x,y)\in\R^2$.
 \end{enumerate}
\end{prop}
The second and main result deals with the spreading in the $x$-direction: we show the existence of an
asymptotic speed of spreading $c^*$ such that the following Theorem holds
 \begin{theorem}\label{spreadingthm}
Let $(u,v)$ be a solution of (\ref{RPeq}) with a nonnegative, compactly supported initial datum $(u_0,v_0)$.
Then, pointwise in $y$, we have: 
 \begin{itemize}
  \item for all $c>c^*$, $\displaystyle\lim_{t\to\infty}\sup_{|x|\geq ct}(u(x,t),v(x,y,t)) = (0,0)$ ;
  \item for all $c<c^*$, $\displaystyle\lim_{t\to\infty}\sup_{|x|\leq ct}\lp(u(x,t),v(x,y,t))-\lp U_s,V_s(y)\rp \rp =(0,0) $.
 \end{itemize}
\end{theorem}
Because $f$ is a KPP-type reaction term,
 it is natural to look for positive solutions of the linearised system  
 \begin{equation}
 \label{RPli}
\begin{cases}
 \partial_t u-D \partial_{xx} u = -\mub u+\int \nu(y)v(t,x,y)dy & x \in \R,\ t>0 \\
 \partial_t v-d\Delta v = f'(0)v +\mu(y)u(t,x)-\nu(y)v(t,x,y) & (x,y)\in \R^2,\ t>0.
\end{cases}
\end{equation}
 We will construct exponential travelling waves and use them 
 to compute the asymptotic speed of spreading in the $x$-direction. Theorem \ref{spreadingthm} relies on the following Proposition:
 \begin{prop}
\label{spreading_speed}
\begin{enumerate}
  \item There exists a limiting velocity $c_*$, depending on $D$ and $d$, such that $\forall c>c^*,\ \exists \la>0,\
 \exists \phi\in H^1(\R)$ positive such that
 $\displaystyle
 (t,x,y)\mapsto e^{-\la(x-ct)} \begin{pmatrix}
                  1 \\ \phi(y)
                 \end{pmatrix}
 $
 is a solution of (\ref{RPli}). No such solution exists if $c<c^*.$
  \item If $D\leq 2d$, then $c_*=c_{KPP}=2\sqrt{df'(0)}$. If $D>2d$, then $c_*>c_{KPP}.$
\end{enumerate}
\end{prop}

Thesee three results easily extend to the two semi-limit models (\ref{RPSL}) and (\ref{RPSL2}). We will develop some proofs only for the 
system (\ref{RPSL}), the other being easier.

\subsubsection{Effect of the nonlocal exchanges on the spreading speed}
Given all these connected models, a natural question is to understand how different exchange terms influence the propagation. 
One possible way to see it is to ask if, with similar parameters, some exchange functions give slower or faster spreading speed than other.
Our results deal with maximal or locally maximal spreading speed. 
Throughout the end of the paper, we consider the set of admissible exchange functions from the road to the field for fixed $\mub$
$$
\Lambda_{\mub} = \{\mu\in C_0(\R),\mu\geq 0, \int\mu=\mub,\mu \textrm{ is even} \}.
$$
Of course, we define $\Lambda_{\nub}$ in a similar fashion. The first result is devoted to the semi-limit case (\ref{RPSL2}),
where the exchange $\nu$ is a Dirac measure at $y=0$, and $\mu$ is nonlocal. For fixed constants $d,D,\nub,\fp,$ for any function $\mu\in\Lambda_{\mub},$ let
$c^*(\mu)$ be the spreading speed associated to the semi-limit system (\ref{RPSL2}) with 
exchange function from the road to the field $\mu.$ Then we have the following property.

\begin{prop}\label{vitessemaxRPSL2}
 Let $c^*_0$ the spreading speed associated with the limit system (\ref{BRReq2}) with the same parameters and exchange rate
 from the road to the field $\mub.$ Then:
 $$
 c^*_0=\sup\{c^*(\mu),\ \mu\in\Lambda_{\mub}\}.
 $$
\end{prop}

The second main result is concerned with the other semi-limit case (\ref{RPSL}), 
where the exchange $\mu$ is a Dirac measure, and $\nu$ is nonlocal ; in our study, 
we consider $\nu$ close to a Dirac measure.
Let the exchange term $\nu$ be of the form
\begin{equation}\label{nuperturbe}
 \nu(y)=(1-\e)\d_0+\e\up(y)
\end{equation}
where 
\begin{equation*}
 \up\in\Lambda_1:=\{\up\in C_0(\R),\up\geq 0, \int\up=1,\up \textrm{ is even} \}.
\end{equation*}

\begin{theorem}\label{thmdenonacceleration}
For some $\up\in\Lambda_1,$ $\e>0,$ let us consider an exchange function of the form (\ref{nuperturbe}). 
Let $c^*(\nu)$ be the spreading speed associated to (\ref{RPSL}) with exchange function $\nu,$ and $c^*_0$ the
one associated to (\ref{BRReq2}) with same parameters.
 There exist $m_1>2$ depending on $\fp$, $M_1$ depending on $\mub$ such that:
 \begin{enumerate}
   \item if $D<m_1$ there exist $\e_0$ and $\up\in\Lambda_1$ such that $\forall \e<\e_0,$ 
  $c^*_0<c^*(\nu);$ 
  \item if  $\mub>4$ and $D,\fp>M_1$ there exists $\e_0$ such that $\forall \up\in\Lambda_1,$ $\forall \e<\e_0,$
$c^*_0>c^*(\nu).$ 
 \end{enumerate}
\end{theorem}

\subsection{Outline and discussion}
The following section is concerned with the Cauchy problem, stationary solutions and the long time behaviour. Its conclusion
is the proof of Proposition \ref{liouville}. The third section is devoted to the proof of Proposition \ref{spreading_speed}, and we prove Theorem \ref{spreadingthm}
in the fourth.
Our results and methods in these two sections shed a new light on those of ~\cite{BRR1} and ~\cite{BRR2}. It is striking
to find the same condition on $D$ and $d$ for the enhancement of the spreading in one direction.  The stationary solutions are
nontrivial and more complicated to bring out. The computation of the spreading speed $c^*$ comes from a nonlinear spectral problem, and 
not from an algebraic system which could be solved explicitly. It also involves some tricky arguments of differential equations.

In the fifth section, we investigate the semi-limit model (\ref{RPSL}). This underlines the robustness of the method for this kind of system.

 We study in the sixth section the asymptotics $D\to+\infty$
in all cases, which has already been done for the initial model in \cite{BRR1}. Such an asymptotics 
for a nonlinearity has also been studied by L. Dietrich in \cite{Dietrich2}.

We prove Proposition \ref{vitessemaxRPSL2} in the seventh section. We show that in the semi-limit case (\ref{RPSL2}), 
the spreading speed is maximal for a concentrate exchange term, that is for the initial limit system (\ref{BRReq2}).
Such a result may be linked to the case of a periodic framework found in \cite{matano}.

It could be expected a similar result in the other semi-limit case (\ref{RPSL}). We prove by two different ways 
that it is not true.
We first investigate the case of a self-similar approximation of a Dirac measure for the nonlocal exchange $\nu.$ 
For these kind of exchange functions, the Dirac measure is a local minimizer for the spreading speed. This is the purpose 
of the eighth section.

Considering that, a natural guess would be that in the semi-limit case (\ref{RPSL}) the Dirac measure is 
a local minimizer anyway. Once again, this is not true. This is the purpose of the last section: 
we prove that any behaviour 
may happen in a neighbourhood of concentrate exchange term.
More precisely, we prove in Theorem \ref{thmdenonacceleration} that if $c^*_0$ is the spreading speed
associated to the limit system (\ref{BRReq2}), considering a perturbated exchange function of the form 
$\nu=(1-\e)\d_0+\e\up,$ that is mainly boundary conditions with a small integral contribution, then
\begin{itemize}	
 \item for some ranges of parameters $D,\mub,\fp,$ in the neighbourhood of $\e=0,$ the maximal speed is $c^*_0;$
 \item for other ranges of these parameters and some integral exchange $\up,$ a perturbation as above enhances the spreading
 for $\e$ small enough.
\end{itemize}
Such a difference between self-similar approximations and general approximations of a Dirac measure may be surprising, 
but a phenomenon of the same kind has already been observed by L. Glangetas in \cite{Glangetas_92} in a totally different context.
We can also notice that these results underline how different are the influences of the two exchange functions.

\paragraph{Acknowledgements} The research leading to these results has received 
funding from the European Research Council under the European Union’s 
Seventh Framework Programme (FP/2007-2013) / ERC Grant Agreement n.321186 - ReaDi -Reaction-Diffusion Equations, Propagation and Modelling.
I am grateful to Henri Berestycki and Jean-Michel Roquejoffre for suggesting me the
model and many fruitful conversations. Part of the work was initiated by a relevant question pointed out by Gr\'egoire Nadin, whom
I thank for it.
I also would like to thank the anonymous referees for their helpful comments.

\section{Stationary solutions and long time behaviour}
In this section, we are concerned with the well-posedness of the system (\ref{RPeq}) combined with the initial
condition
\begin{equation}\label{initial}
\begin{cases}
u|_{t=0}=u_0 & \in \R \\
v|_{t=0}=v_0 & \in \R^2.
\end{cases}
\end{equation} 

\subsection{Existence, uniqueness and comparison principle}
The system (\ref{RPeq}) is standard, in the sense that the coupling does not appear in the diffusion nor the reaction term.
Anyway, well-posedness still has to be mentioned.

\begin{prop}\label{Cauchy}
Under the above assumptions on $f$, $\mu$ and $\nu$, the Cauchy problem (\ref{RPeq})-(\ref{initial}) admits 
a unique nonnegative bounded solution.
\end{prop}

Using the formalism of ~\cite{henry}, it is easy to show that the linear part on (\ref{RPeq}) defines a sectorial operator, and that the non-linear is globally
Lipschitz on $X:=\Cu(\R)\times\Cu(\R^2)$, which gives the existence and uniqueness of the solution of (\ref{RPeq}).

We can also derive the uniqueness of the solution of (\ref{RPeq}) by showing that comparison between subsolutions
and supersolutions is preserved during the evolution. Moreover, the following property will also be the key point in our 
later study of the spreading. Throughout this article, we will  call a subsolution (resp. a supersolution) a couple satisfying 
the system (in the classical sense) with the equal signs replaced by $\leq$ (resp. $\geq$) signs, which is also continuous
up to time 0.
\begin{prop}
 \label{comparaison}
Let $(\su,\sv)$ and $(\us,\vs)$ be respectively a subsolution bounded from above and a supersolution
bounded from below of (\ref{RPeq}) satisfying $\su\leq\us$ and $\sv\leq\vs$ at $t=0$. Then, either $\su<\us$ and
$\sv<\vs$ for all $t>0$, or there exists $T>0$ such that $(\su,\sv)=(\us,\vs),\ \forall t\leq T.$
\end{prop}

Once again, the proof is quite classical and omitted here.
This comparison principle extend immediately to generalised sub and supersolutions given by the supremum of subsolutions 
and the infimum of supersolutions. For our spreading result, we will need a more general class of subsolutions, already used 
for several results in this context. 
 See for instance Proposition 3.3 in \cite{BRR1}. 

\subsection{Long time behaviour and stationary solutions}
The main purpose of this section is to prove that any (nonnegative) solution of (\ref{RPeq}) converges locally uniformly to
a unique stationary solution $(U_s,V_s)$, which is bounded, positive, $x$-independent, and solution of the 
stationary system of equations (\ref{stateq}): 
\begin{equation}
\label{stateq}
\begin{cases}
-DU''(x)  & =  -\mub U(x)+\int \nu(y)V(x,y)dy \\
-d\Delta V(x,y)& =  f(V)+\mu(y)U(x)-\nu(y)V(x,y).
\end{cases}
\end{equation}
In the same way as above, we call a subsolution (resp. a supersolution) of (\ref{stateq}) a couple satisfying the 
system (in the classical sense) with the equal signs replaced by $\leq$ (resp. $\geq$). 

\begin{prop}
 \label{stationary1}
Let $(u,v)$ be the solution of (\ref{RPeq}) starting from $(u_0,v_0)\not\equiv(0,0)$. then there exist two positive, bounded, 
x-independent, stationary solutions $(U_1,V_1)$ and $(U_2,V_2)$ such that  
$$
U_1\leq \underset{t\to +\infty}{\liminf}\ u(x,t) \leq \underset{t\to +\infty}{\limsup}\ u(x,t)\leq U_2,
$$
$$
V_1(y)\leq \underset{t\to +\infty}{\liminf}\ v(x,y,t) \leq \underset{t\to +\infty}{\limsup}\ v(x,y,t)\leq V_2(y),
$$
locally uniformly in $(x,y)\in \R^2$.
\end{prop}

\begin{proof}
The proof is adapted from ~\cite{BRR2}. We first need a $L^\infty$ a priori estimate.
\paragraph{A priori estimate}
Considering the hypothesis on the reaction term $f$, there exists $K>0$ such that
$$
\forall s\geq K,\ f(s)\leq s(\frac{\nub}{\mub}\mu(y)-\nu(y)),\ \forall y\in \R.
$$
Thus, for all constant $V\geq K$, $V(\frac{\nub}{\mub},1)$ is a supersolution of (\ref{RPeq}).

\paragraph{Construction of $(U_1,V_1)$}
Let $R>0$ large enough in such a way that the first eigenvalue of the Laplace operator with Dirichlet boundary condition in
$B_R\subset \R^2$ is less than $\frac{f'(0)}{3d}$, $\phi_R$ the associated eigenfunction. We extend $\phi_R$ to 0 outside $B_R$.
$\phi_R$ is continuous, bounded, and satisfies
$$
-d\Delta\phi_R\leq\frac{1}{3}f'(0)\phi_R \ \textrm{in}\ \R^2.
$$ 
Let us choose $\e>0$ such that if $0<x\leq \e,\ f(x)>\frac{2}{3}f'(0)x$. Then define
$M>R$ such that $\forall y \ /\ |y|>M-R,\ \nu(y)\leq\frac{1}{3}f'(0)$.
Since $(u_0,v_0)\not\equiv(0,0)$ and $(0,0)$ is a solution, the comparison principle
implies that $u,v>0,\ \forall t>0.$
Now, let us define $\eta$ such that $\eta\phi_R(x,|y|-M)<v(x,y,1)$ and $\eta \|\phi_R\|_{\infty}\leq\e.$
Define $\underline{V}(x,y):=\eta\phi_R(x,|y|-M)$, and, up to a smaller $\eta,$ $(0,\underline{V})$ is a subsolution of (\ref{RPeq}) which is strictly below
$(u,v)$ at $t=1$. Let $(u_1,v_1)$ be the solution of (\ref{RPeq}) starting from $(0,\underline{V})$ at $t=1$;
 $(u_1,v_1)$ is strictly increasing in time, bounded by $K(\frac{\nub}{\mub},1)$,
and converges to a positive stationary solution $(U_1,V_1)$, satisfying
$$
U_1\leq \underset{t\to +\infty}{\liminf}\ u  \qquad V_1\leq \underset{t\to +\infty}{\liminf}\ v
$$
locally uniformly in $(x,y)\in\R^2$.

It remains to show that $(U_1,V_1)$ is invariant in $x$. For $h \in \R$, let us denote $\tau_h$ the translation by $h$ in the x-direction: 
$\tau_h w(x,y)=w(x+h,y)$. Since $\underline{V}$ is compactly supported, there exists $\e>0$ such that 
$$
\forall h\in (-\e,\e),\ \tau_h\underline{V}<V_1 \textrm{ and } \tau_h\underline{V}<v \textrm{ at } t=1.
$$
Thus, because of the $x$-invariance of the system (\ref{RPeq}), the solution $(\ut_1,\vt_1)$ of (\ref{RPeq}) 
starting from $(0,\tau_h\underline{V})$ at $t=1$ is equal to the translated $(\tau_h u_1, \tau_h v_1)$.
So, $(\ut_1,\vt_1)$ converges to $(\tau_h U_1,\tau_h V_1)$. But, since $(\ut_1,\vt_1)$ is below $(U_1,V_1)$ at $t=1$
and $(U_1,V_1)$ is a (stationary) solution, from the comparison principle given by Proposition \ref{comparaison} we deduce
$(\ut_1,\vt_1)<(U_1,V_1),\ \forall t >1$, and then
$$
(\tau_h U_1,\tau_h V_1)\leq (U_1,V_1),\ \forall h\in (-\e,\e).
$$
Namely, $(U_1,V_1)$ does not depend on $x$.

\paragraph{Construction of $(U_2,V_2)$}
Let $\overline{V}=\max (\|v_0\|_\infty,K)$ and $\overline{U}=\max (\|u_0\|_\infty,\overline{V}\frac{\nub}{\mub})$. Let $(u_2,v_2)$ be the solution 
of (\ref{RPeq}) with initial datum $(\overline{U},\overline{V})$. From the comparison principle (\ref{comparaison}), 
$(u,v)$ is strictly below $(u_2,v_2)$, for all $t>0,\ (x,y)\in\R^2$. Moreover, since $(\overline{U},\overline{V})$ is a supersolution 
of (\ref{RPeq}) it is clear that $\partial_t u_2,\ \partial_t v_2 \leq 0$ at $t=0$. Still using 
Proposition \ref{comparaison}, it is true for all $t\geq 0$, and $u_2$ and $v_2$ are nonincreasing in $t$, bounded from below by $(U_1,V_1)$.
Thus, $(u_2,v_2)$ converges as $t\to \infty$ to a stationary solution $(U_2,V_2)$ of (\ref{RPeq}) satisfying
$$
\underset{t\to +\infty}{\limsup}\ u(t,x)\leq U_2 \qquad \underset{t\to +\infty}{\limsup}\ v(t,x,y)\leq V_2(y),
$$
locally uniformly in $(x,y)\in\R^2$. From the construction of $(U_2,V_2)$, which is totally independent of the $x$-variable, it
is easy to see that $(U_2,V_2)$ does not depend in $x$.
\end{proof}

\paragraph{Uniqueness of the stationary solution}
The previous proposition provides a theoretical proof of the existence of stationary solutions. It also means that a solution is either
converging to a stationary solution, or will remain between two stationary solutions. In order to obtain a more precise description 
of the long time behaviour, we need the following uniqueness result.

\begin{prop}
\label{stationary2}
 There is a unique positive, bounded, stationary solution of (\ref{RPeq}), denoted $(U_s,V_s)$.
\end{prop}

To prove the uniqueness, we first need the following intermediate lemma which is the key to all uniqueness properties in this kind of 
problem. The idea that a bound from below implies uniqueness appeared for the first time in ~\cite{BHR}.

\begin{lem}
\label{stationary3}
Let $(U,V)$ be a positive, bounded stationary solution of (\ref{RPeq}). Then there exists $m>0$ such that 
$$
\forall (x,y)\in \R^2,\ U(x)\geq m,\ V(x,y)\geq m.
$$
\end{lem}

\begin{proof}
Let $(U,V)$ be such a stationary solution.

\textit{First step}: there exists $M>0$ such that 
$$m_1:=\inf \{V(x,y),\ |y|>M\}>0.$$
We will state the proof for positie $y.$
Let $R>0$ large enough in such a way that the first eigenvalue of the Laplace operator with Dirichlet boundary condition in
$B_R\subset \R^2$ is less than $\frac{f'(0)}{3d}$, $\phi_R$ the associated eigenfunction. We extend $\phi_R$ to 0 outside $B_R$.
$\phi_R$ is continuous, bounded in $\R^2$, positive in $B_R$. For $M>0,$ we define $\tau_M\phi_R(x,y)=\phi_R(x,y-M).$ 
As above, let us define $M_0>R$ such that $\forall y \ /\ |y|>M_0-R,\ \nu(y)\leq\frac{1}{3}f'(0).$ Then, there exists $\e>0$ 
such that $\forall M>M_0,$ $\lp 0,\e\tau_M\phi_R\rp$ is a subsolution of (\ref{stateq}). As $V$ is positive, up to smaller $\e,$
we can suppose that $\e\tau_{M_0}\phi_R<V.$ Now, we claim that 
$$
\forall y>M_0,\ V(0,y)>\e\phi_R(0,0).
$$
Indeed, let us define 
$$
M_1:=\sup \{M\geq M_0,\ \forall K\in [M_0,M],\ \e\tau_K\phi_R<V\}.
$$
Since $V$ and $\phi_R$ are continuous, $M_1>M_0.$ Suppose that $M_1<+\infty.$ Then $\lp U,V\rp\geq \lp0,\e\tau_{M_1}\phi_R\rp$ and there exists $(x_0,y_0),$ 
$V(x_0,y_0)=\e\tau_{M_1}\phi_R(x_0,y_0).$ Considering that the dynamical system starting from $(0,\e\tau_{M_1}\phi_R),$ which is a subsolution, 
we get a contradiction from Proposition \ref{comparaison}.
Hence $M_1=+\infty$ and our claim is proved. Using the same argument in the $x$-direction, we get that $m_1\geq \e\phi_R(0,0).$

 
 \textit{Second step}: 
 $$
 m_2:=\inf \{V(x,y),\ (x,y)\in\R^2\}>0.
 $$
 If $m_2=m_1$, the assumption is proved. It is obvious that $m_2\geq0.$ Let us assume by way of contradiction that $m_2=0.$ We consider 
 $(x_n,y_n)$ such that $V(x_n,y_n)\to 0$ with $n\to \infty$. Now, we set 
 $$
 U_n:=U(.+x_n),\ V_n:=V(.+x_n,.+y_n),\ \mu_n:=\mu(.+y_n),\ \nu_n:=\nu(.+y_n).
 $$
 Using the fact that $U$ and $V$ are smooth and bounded, by standard elliptic estimates (see ~\cite{GT} for example), there exists
 $\vp:\N\to\N$ strictly increasing such that $(U_{\vp(n)})_n,\ (V_{\vp(n)})_n$ converge locally uniformly to 
 some functions $\tilde{U},\ \tilde{V}$ satisfying 
 $$
 \begin{cases}
-D\tilde{U}''(x)  & =  -\mub \tilde{U}(x)+\int \tilde{\nu}(y)\tilde{V}(x,y)dy \\
-d\Delta \tilde{V}(x,y) & =  f(\tilde{V})+\tilde{\mu}(y)\tilde{U}(x)-\tilde{\nu}(y)\tilde{V}(x,y)
\end{cases}
 $$
where $\tilde{\mu},\ \tilde{\nu}$ are some translated of $\mu,\ \nu$. Furthermore,  $\tilde{V}\geq 0$ and $\tilde{V}(0,0)=0$. Thus 
 in a neighbourhood of $(0,0)$ we have 
$$
-d\Delta \tilde{V}(x,y)+\tilde{\nu}(y)\tilde{V}(x,y)\geq 0,\ \min(\tilde{V})=0. 
$$
From the strong elliptic maximum principle, we deduce $\tilde{V}\equiv 0.$ But by step 1 $\tilde{V}(.,2M)\geq m_1>0$, and we get a contradiction.
Hence the result stated above, $m_2:=\inf(V)>0.$ 

Third step: $U$ is also bounded from below by a positive constant. Indeed, if we set $\phi(x)=\frac{1}{D}\int \nu(y)V(x,y)dy$, $U$ is 
solution of 
\begin{equation}
\label{eqU}
-U''+\frac{\mub}{D}U=\phi,
\end{equation}
with $\phi$ continuous and $\phi\geq m_2\|\nu\|_{L^1}$. Using $\Phi(x)=\frac{D}{2\mub}\exp(-\sqrt{\frac{\mub}{D}}|x|)$
which is the fundamental solution of (\ref{eqU}) we get 
$$
U(x)=\phi*\Phi(x)\geq \|\Phi\|_{L^1}.m_2.\|\nu\|_{L^1}:=m_3>0.
$$
Now, set $m=\inf(m_1,m_2,m_3)$ and the proof is concluded.
\end{proof}

\paragraph{Proof of proposition \ref{stationary2}}
It remains now to prove the uniqueness of the stationary solution of (\ref{RPeq}). The difficulties come from the fact that it is a 
coupled system in an unbounded domain: for bounded domains, uniqueness was proved in ~\cite{Berestycki1}. Let $(U_1,V_1)$, $(U_2,V_2)$ be two 
bounded, positive solutions of (\ref{stateq}), and let us show that $(U_1,V_1)=(U_2,V_2).$ From Lemma \ref{stationary3}, there exists 
$m>0$ such that $(U_i,V_i)\geq m,\ i=1..2.$ Hence, for $T$ large enough, $T(U_1,V_1)>(U_2,V_2).$ Let 
$$
T_1=\inf\{T,\ \forall T'>T,\ T'(U_1,V_1)>(U_2,V_2)\}>0,
$$
and
$$
(\d U,\d V)=T_1(U_1,V_1)-(U_2,V_2).
$$
Up to take $T_1(U_2,V_2)-(U_1,V_1)$ if needed, we can suppose $T_1\geq1.$ The couple $(\d U,\d V)$ satisfies the following system: 
\begin{equation*}
\begin{cases}
-D\d U''(x)  & =  -\mub \d U(x)+\int \nu(y)\d V(x,y)dy \\
-d\Delta \d V(x,y)& =  T_1 f(V_1)-f(V_2)+\mu(y)\d U(x)-\nu(y)\d V(x,y)
\end{cases}
\end{equation*}
and $\inf(\d U)=0\ \underline{\textrm{or}} \ \inf(\d V)=0.$
In order to show that $(\d U,\d V)\equiv 0$ we have to distinguish five cases.

Case 1: there exists $(x_0,y_0)\in\R^2,\ \d V(x_0,y_0)=0.$ Then, using the fact that
$f(0)=0$ and that $f$ is strictly concave, we can easily check that $T_1f(V_1)-f(V_2)\geq 0$ in a 
neighbourhood of $(x_0,y_0).$ Thus, because $\d U\geq 0$, $\d V$ is solution of the
inequality system
$$
\begin{cases}
-d\Delta \d V +\nu \d V & \geq 0 \\
\d V \geq 0, & \d V(x_0,y_0)=0.
\end{cases}
$$
From the elliptic maximum principle, we infer $\d V\equiv 0$. Because $\mu\not\equiv 0,$ we immediately get $\d U\equiv 0.$
So $(U_2,V_2)=T_1(U_1,V_1)$ ; subtracting the two systems (\ref{stateq}) in $(U_1,V_1)$ and $T_1(U_1,V_1)$ yields 
$T_1 f(V_1)=f(V_1)$ and $V_1>0$. So $T_1=1$, and $(U_2,V_2)=(U_1,V_1)$.

Case 2: there exists $x_0$ such that $\d U(x_0)=0.$ In the same way we infer $\d U\equiv 0$. Then, $\forall x\in\R,\ \int\nu\d V=0.$ In particular, 
there exists $y_0$ such that $\d V(x_0,y_0)=0$, and the problem is reduced to the (solved) first case: $T_1=1$, and $(U_2,V_2)=(U_1,V_1)$.

Case 3: there is a contact point for $U$ at infinite distance. Formally, there exists $(x_n)_n,\ |x_n|\to\infty$ such that
$\d U(x_n)\to 0$ with $n\to\infty.$ We set
$$
U_i^n:=U_i(.+x_n),\ V_i^n:=V_i(.+x_n,.),\ i=1,2.
$$
In the same way as above, there exist $\tilde{U}_i,\ \tilde{V}_i$ such that, up to a subsequence, $(U_i^n,V_i^n)$ converges locally uniformly
to $(\tilde{U}_i,\ \tilde{V}_i)$, and the couples
$(\tilde{U}_1,\tilde{V}_1)$ and $(\tilde{U}_2,\tilde{V}_2)$ both satisfy (\ref{stateq}) and 
$$\begin{cases}
T_1=\inf\{T,\ \forall T'>T,\ T'(\tilde{U}_1,\tilde{V}_1)>(\tilde{U}_2,\tilde{V}_2)\}, \\
(T_1 \tilde{U}_1-\tilde{U}_2)(0)=0.
\end{cases}
$$
The problem is once again reduced to the first case, and $T_1=1.$

Case 4: there is a contact point for $V$ at infinite distance in $x$, finite distance in $y$, say $y_0$. We use the same trick as above, 
the limit problem is this time reduced to the second case, and we still get $T_1=1.$

Case 5: there is a contact point for $V$ at infinite distance in $y$. That is to say there exist $(x_n)_n,\ (y_n)_n$, with
$|y_n|\to \infty$ such that $\d V (x_n,y_n)\underset{n\to\infty}{\longrightarrow}0.$ Once again, we set
$$
V_i^n:=V_i(.+x_n,.+y_n),\ i=1,2.
$$
Now, considering that $U_1,\ U_2$ are bounded and that $\mu,\nu\underset{|y|\to\infty}{\longrightarrow}0$, $(V_1^n)_n,\ (V_2^n)_n$ converge
locally uniformly to some functions $\tilde{V}_1,\ \tilde{V}_2$ which satisfy
$$
\begin{cases}
 -d\Delta \tilde{V}_i & = f(\tilde{V}_i) \\
 (T_1 \tilde{V}_1-\tilde{V}_2)(0,0) & = 0
\end{cases}
$$
and $(T_1 \tilde{V}_1-\tilde{V}_2) \geq 0$ in a neighbourhood of $(0,0)$. Thus, using the concavity of $f$ as in the first case, we get $T_1=1.$

From the five cases considered above, whatever may happen, $T_1 = 1,$ 
and the proof is complete.
 \qed

The proof of Proposition \ref{liouville} is now a consequence of Propositions \ref{stationary1} and \ref{stationary2}.

\section{Exponential solutions of the linearised system}

Looking for supersolution of the system (\ref{RPeq}) lead us to search
positive solutions of the linearised system (\ref{RPli}), hence we are looking for solutions of the form: 
\begin{equation}
\label{solexp} 
\begin{pmatrix}
 u(x,t) \\
v(x,y,t)
\end{pmatrix}
= e^{-\la(x-ct)} \begin{pmatrix}
                  1 \\ \phi(y)
                 \end{pmatrix},
\end{equation}
where $\la, c$ are positive constants, and $\phi$ is a nonnegative function in $H^1(\R)$.
The system on $(\la, \phi)$ reads: 
\begin{equation}
\label{eq1}
\begin{cases}
 -D\la^2+\la c+\mub  = \int \nu(y)\phi(y)dy \\
-d\phi''(y)+(\la c-d\la^2-f'(0)+\nu(y))\phi(y)  =  \mu(y).
\end{cases}
\end{equation}

The first equation of (\ref{eq1}) gives the graph of a function $\la\mapsto\Psi_1(\la,c):= -D\la^2+\la c+\mub$, which,
if (\ref{solexp}) is a solution of (\ref{RPli}), is equal to $\int \nu(y)\phi(y)dy$. \\
The second equation of (\ref{eq1})
gives, under some assumptions on $\la$, a unique solution $\phi=\phi(y;\la,c)$ in $H^1(\R)$. To this unique solution
we associate the function $\Psi_2(\la,c):=\int \nu(y)\phi(y)dy$.
Let us denote $\Gamma_1$ the graph of $\Psi_1$ in the $(\la, \Psi_1(\la))$ plane, and $\Gamma_2$ the graph
of $\Psi_2$. So, (\ref{eq1}) amounts to the investigation of $\la,\ c>0$ such that $\Gamma_1$ and $\Gamma_2$ intersect.


The graph of $\la \mapsto \Psi_1(\la)$ is a parabola. As we are looking for a nonnegative function $\phi$, we are interested
in the positive part of the graph. The function $\la\mapsto\Psi_1(\la)$ is nonnegative for $\la \in [\la_1^-(c) , \la_1^+(c)]$, 
with $\la_1^\mp(c) = \frac{c\mp\sqrt{c^2+4D\mub}}{2D}$.
\\
It reaches its maximum value in $\la=\frac{c}{2D}$, with $\Psi_1(\frac{c}{2D})=\mub+\frac{c^2}{4D} > \mub$.\\
We also have 
\begin{equation}
 \Psi_1(0)=\Psi_1(\frac{C}{D})=\mub,
\end{equation}
which will be quite important later.
\\
We may observe that: with $D$ fixed, $(\la_1^-(c) , \la_1^+(c)) \underset{c\to +\infty}{\longrightarrow} (0^-, +\infty)$;
$\la \mapsto \Psi_1(\la)$ is strictly concave;
$\displaystyle\frac{d\Psi_1}{d\la}_{|\la=c/D}=-c$.We can summarize it in fig. (\ref{parabole}).

\begin{figure}[ht]
   \centering
   \def\mub{\overline{\mu}}
\psset{unit=0.7}
\psset{linewidth=0.03}

\begin{pspicture}(-3,-1)(9,8)
\psparabola[linewidth=0.1,linecolor=red](-2,0)(3,6.25)
\psline{->}(-3,0)(9,0)
\psline{->}(0,-1)(0,7.5)
\psline[linestyle=dashed](0,4)(6,4)
\psline[linestyle=dashed](3,6.25)(3,0)
\psline[linestyle=dashed](0,6.25)(3,6.25)
\psline[linestyle=dashed](6,4)(6,0)
\uput{0.2}[180](0,6.25){$\mub+\frac{c^2}{4D}$}
\uput{0.2}[180](0,4){$\mub$}
\uput{0.2}[260](3,0){$\frac{c}{2D}$}
\uput{0.2}[260](6,0){$\frac{c}{D}$}
\uput{0.2}[260](8,0){$ \la_1^+$}
\uput{0.2}[225](0,0){0}
\uput{0.2}[0](9,0){$\lambda$}
\uput{0.2}[0](0,7.5){$\Psi_1(\lambda)$}
\end{pspicture}
   \caption{\label{parabole}representation of $\Gamma_1$}
\end{figure}
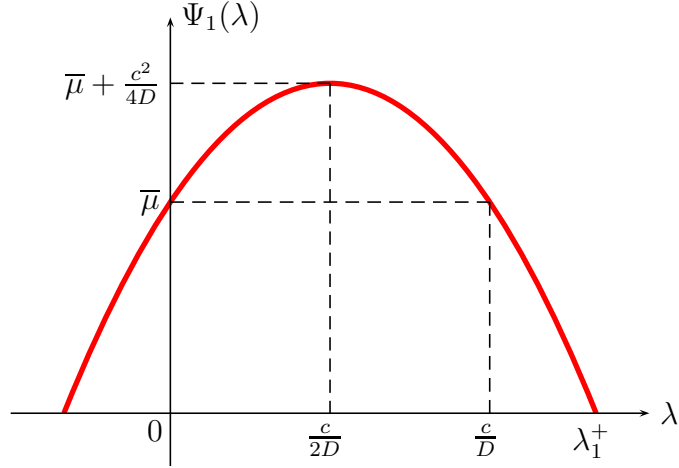

\subsection{Study of $\Psi_2$}
The study of $\Psi_2$ relies on the investigation of the solution $\phi=\phi(\la;c)$ of
\begin{equation}
\label{eqgeneralesurphi}
\begin{cases}
 -d\phi''(y)+(\la c-d\la^2-f'(0)+\nu(y))\phi(y)  =  \mu(y) \\
 \phi \in H^1(\R) \ \phi\geq0. 
\end{cases}
\end{equation}
Since $\mu$ is continuous and decays no slower than an exponential, $\mu$ belongs to $L^2(\R)$. Since $\nu$ is nonnegative and 
bounded, the Lax-Milgram theorem assures us that (\ref{eqgeneralesurphi}) admits a unique solution if
$\la c-d\la^2-f'(0) >0$, that is to say if $\la$ belongs to $]\la_2^-(c),\la_2^+(c)[$, where
$$
\la_2^\mp(c) = \frac{c\mp\sqrt{c^2-c_{KPP}^2}}{2d},
$$
with
$$
c_{KPP}=2\sqrt{df'(0)}.
$$
As in ~\cite{BRR1}, the KPP-asymptotic spreading speed will have a certain importance in the study of the spreading
in our model. Moreover, since $\nu, \mu$ tend to $0$ with $|y|\to\infty$, an easy computation will show that, for
$\la<\la_2^-$ or $\la>\la_2^+$, equation (\ref{eqgeneralesurphi}) cannot have a constant sign solution. Morever, we look for $H^1$ solutions. We will see in Lemma \ref{lemhomo}
that it prevents the existence of a solution for $c=c_{KPP}.$
Thus,
\begin{equation}
\label{existencegamma2}
 \Gamma_2 \textrm{ exists if and only if } c>c_{KPP}.
\end{equation}

 The main properties of $\Psi_2$ are the following: 

\begin{prop}\label{sourire}
 If $c>c_{KPP}$, then: 
\begin{enumerate}
 \item $\la \mapsto \Psi_2(\la)$ defined on $]\la_2^-,\la_2^+[$ is positive, smooth, strictly convex and
symmetric with respect to the line $\{\la=\frac{c}{2d}\}$. With $\la$ fixed we also have $\frac{d}{dc}\Psi_2(\la;c)<0$.
 \item $\Psi_2(\la) \underset{\la\to \la_2^\mp}{\longrightarrow} \mub$.
 \item $\frac{d\Psi_2}{d\la}(\la)\underset{\underset{\la>\la_2^-}{\la\to \la_2^-}}{\longrightarrow} -\infty$.
\end{enumerate}
\end{prop}
The graph $\Gamma_2$ looks like fig. (\ref{Sourire}).
\begin{figure}[ht]
   \centering
   \def\mub{\overline{\mu}}

\begin{pspicture}(-2,-1)(8,5)

\psellipticarc[linewidth=0.1,linecolor=red](4,4)(3,2){180}{360}
\psline{->}(-2,0)(8,0)
\psline{->}(0,-1)(0,5)
\psline[linestyle=dashed](0,4)(7,4)
\uput{0.2}[180](0,4){$\mub$}
\psline[linestyle=dashed](1,0)(1,4)
\uput{0.2}[270](1,0){$\lambda_2^-$}
\psline[linestyle=dashed](7,0)(7,4)
\uput{0.2}[270](7,0){$\lambda_2^+$}
\psline[linestyle=dashed](4,0)(4,2)
\uput{0.2}[270](4,0){$\frac{c}{2d}$}
\uput{0.2}[0](8,0){$\lambda$}
\uput{0.2}[0](0,5){$\Psi_2(\lambda)$}
\end{pspicture}
   \caption{\label{Sourire}representation of $\Gamma_2$}
\end{figure}
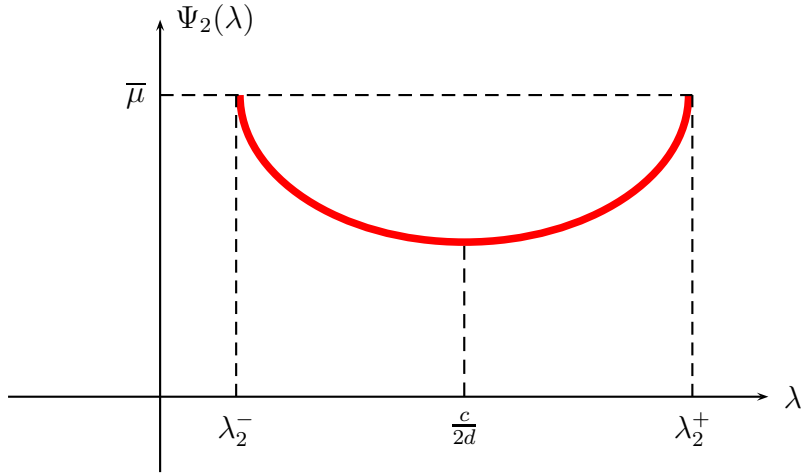

{\textbf{Proof of the first part of proposition (\ref{sourire})}}

\paragraph{Positivity, smoothness}
For all $\la$ in $]\la_2^-,\la_2^+[$, 
\begin{equation}
P(\la):=\la c-d\la^2-f'(0) >0.
\end{equation}
Consequently, $\forall \la \in ]\la_2^-,\la_2^+[, \ \forall y\in \R,\ P(\la)+\nu(y)>0$. From the elliptic maximum principle, as
$\mu$ is nonnegative,
we deduce that $\phi(y)>0,\ \forall y \in \R$. Hence, since $\nu$ is nonnegative, we have $\Psi_2(\la)=  \int \phi(y;\la)\nu(y)dy >0$,
and $\Psi_2$ is positive.
\\
Considering that $\la\mapsto P(\la)$ is polynomial, with the analytic implicit function theorem,
we see immediately 
that $\la\mapsto \phi(y;\la)$ is analtytic (see \cite{Cartan}, Theorem 3.7.1). Since $\nu$ is integrable, $\la\mapsto\Psi_2(\la)$
is also analytic.
\\
From the symmetry of $\la\mapsto P(\la)$ and the uniqueness of the solution, we deduce the symmetry of $\Gamma_2$ with respect 
to the line $\{\la=\frac{c}{2d}\}$.

\paragraph{Monotonicity, convexity}
Denote by $\pl$ the derivative of $\phi$ with respect to $\la$. Then, if we differentiate (\ref{eqgeneralesurphi}) with respect to $\la$, we can see that 
$\pl$ satisfies: 
\begin{equation}
 \label{eqpl}
-d\pl''(y)+(P(\la)+\nu(y))\pl(y)  =  (2d\la-c)\phi(y).
\end{equation}
In the same way as equation (\ref{eqgeneralesurphi}), (\ref{eqpl}) has a unique solution in $H^1(\R)$ for all $\la \in ]\la_2^-,\la_2^+[$.
Since $\phi$ is positive, $\pl$ is of constant sign, with the sign of $(2d\la-c)$.
Hence we have that $\Psi_2$ is decreasing on $]\la_2^-,\frac{c}{2d}[$ and increasing on $]\frac{c}{2d},\la_2^+[$. 
\\
Differentiating once again (\ref{eqpl}) with respect to $\la$, the second derivative of $\phi$ with respect to $\la$
satisfies:  
\begin{equation}
 \label{eqpll}
-d\pll''(y)+(P(\la)+\nu(y))\pll(y)  =  2d\phi(y)+2(2d\la-c)\pl(y).
\end{equation}
In the same way, $\phi$ is positive for all $\la\in ]\la_2^-,\la_2^+[$, and $\pl(\la)$ has the positivity of $(2d\la-c)$. Hence
the left term of equation (\ref{eqpll}) is positive, for all $\la\in ]\la_2^-,\la_2^+[$, and $\Psi_2$ is strictly
convex on $]\la_2^-,\la_2^+[$.
\\
With the same arguments we see that $\phi_c$, the derivative of $\phi$ with respect to $c$, satisfies
\begin{equation*}
 -d\phi_c''+(P(\la)+\nu)\phi_c = -\la\phi <0,
\end{equation*}
and then we get $\int_{\R}\phi_c(y) \nu(y)dy = \frac{d}{dc}\Psi_2(\la;c)<0.$


In order to end the proof of the proposition (\ref{sourire}), we need to study 
behaviour of $\Psi_2$ near $\la_2^-.$ Setting $\e=P(\la),$ it is sufficient to study the behaviour of the solution $\phi=\phi(y;\e)$
of 
\begin{equation}
\label{epseq1}
\begin{cases}
   -\phi''(y)+(\e+\nu(y))\phi(y)=\mu(y) \\
  \phi\in H^1(\R),\ \e>0,\ \e\to 0.
 \end{cases}
\end{equation}
The main lemma here is the following, which will evidently conclude Proposition \ref{sourire}: 
\begin{lem}\label{lemeps}
\begin{enumerate}
 \item If $\phi$ is solution of (\ref{epseq1}) then $\int_{\R}\phi(y)\nu(y)dy 
\underset{\underset{\e>0}{\e\to 0}}{\longrightarrow} \mub$ holds true. 
Moreover, $\|\phi\|_{L^{\infty}}$ is uniformly bounded on $\e$.
\item The derivative of $\phi$ with respect to $\e$, denoted $\phi_{\e}$, satisfies $\int_{\R}\phi_{\e}(y)\nu(y)dy
\underset{\underset{\e>0}{\e\to 0}}{\longrightarrow} -\infty. $
\end{enumerate}
 \end{lem}

\paragraph{Proof of the first part of the Lemma \ref{lemeps}}
An explcit computation is needed. We use a boxcar function for this. Under the assumptions on $\nu$ and $\mu$, there exist $\a,\ M,\ m_1>0$ such that: 
\begin{itemize}
 \item $\nu(y) \geq \a \mathbf{1}_{[-m_1,m_1]},\ \forall y\in \R$ (because $\nu(0)>0$, and $\nu$ is continuous);
 \item $\mu(y)\leq M e^{-a|y|},\ \forall y\in \R$ (from the exponential decay of $\mu$).
\end{itemize}

Denoting $\psi=\psi(y;\e)$ the solution of
\begin{equation}
 \label{eqpsi1}
-\psi''+(\e+\a \mathbf{1}_{[-m_1,m_1]})\psi=M e^{-a|y|},
\end{equation}
$\psi$ is a supersolution for (\ref{epseq1}) and
\begin{equation}
 \label{majphipsi1}
\forall \e >0,\ \forall y\in \R,\ 0<\phi(y;\e)\leq\psi(y;\e).
\end{equation}
We have already seen that $\forall \e>0,\ \int_{\R}\phi''(y;\e)dy=0$. Consequently, the assumption  
$\int_{\R}\phi(y)\nu(y)dy\underset{\e\to 0}{\longrightarrow} \mub$ is equivalent to 
$\e\int_{\R}\phi(y;\e)dy\underset{\e\to 0}{\longrightarrow} 0$.
To conclude, it remains to compute the solution $\psi$ and to show that
$\e\int_{\R}\psi(y;\e)dy\underset{\e\to 0}{\longrightarrow} 0$. But the solution of (\ref{eqpsi1}) can be explicitly computed, which gives that
$\|\phi(\e)\|_{L^{\infty}(\R)}$ is uniformly bounded on $\e$ and that ther exists $C>0$ such that
for $\e>0$ small and $y>m_1$, 
$$
\psi(y;\e)< C e^{-\sqrt{\e}y},
$$
so
$$
\int_\R \psi(y;\e)dy = O(\frac{1}{\sqrt{\e}}) \textrm{ as } \e\to 0
$$
and
$$
\e\int_\R \psi(y;\e)dy \underset{\e\to 0}{\longrightarrow}0,
$$
which concludes the proof of the first statement in Lemma \ref{lemeps}. Notice that we also get that there exist two constant $C_1,C_2$ not depending
on $\e$ such that
for all $y$ in $\R$, $\psi(y;\e)\leq C_1 e^{-\sqrt{\e}|y|}+C_2 e^{-a|y|},$ that will be useful later.
\qed
\bigbreak

Let us prove the second part of Lemma (\ref{lemeps}). 
In order to prove it, we will first deal with the study of the homogeneous limit differential equation.

\begin{lem}
 \label{lemhomo}
 Let us consider the scalar homogeneous equation (\ref{homogeneous}): 
 \begin{equation}
  \label{homogeneous}
  -\psi''+\nu.\psi=0.
 \end{equation}
Under the assumptions on $\nu$, there exist $\phi_1$, $\phi_2$ satisfying
\begin{itemize}
 \item $\phi_1(x)\underset{x\to+\infty}{\longrightarrow}0$, and, for x large enough, $\phi_1(x)\geq0$ ;
 \item $\exists C_1,C_2>0$ such that $C_1x \leq \phi_2(x)\leq C_2x$ when $x$ goes to $+\infty$ (notation: $\phi_2(x)=\varTheta(x)$ ) ;
\end{itemize}
such that
$$
\begin{cases}
\psi_1:=1+\phi_1 \\
\psi_2:=\phi_2(1+\phi_1)
\end{cases}
$$
is a fundamental system of solutions of (\ref{homogeneous}).
\end{lem}

\begin{proof}
 {\bf Construction of $\phi_1$}: let $\psi:=1+\phi_1$ be a solution of (\ref{homogeneous}). Thus, $\phi_1$ must satisfy
 \begin{equation}
 \label{homogeneous2}
 -\phi_1''+\nu+\nu.\phi_1 = 0. 
 \end{equation}
 Let us show that there exists a solution of (\ref{homogeneous2}) which is nonnegative for $x$ large enough and tends to $0$ as $x$ goes to $+\infty$.
 Let $M\geq0$ such that $\int_M^{\infty}\int_x^{\infty}\nu(y)dydx<1$ which is possible thanks to the assumption (\ref{nucond}) on $\nu$. 
Now, define
 $$
 \mathcal{E}:=\{\phi\in C([M,+\infty[)/ \forall x\geq M,\phi(x)\geq0 \textrm{ and }\phi(x)\underset{x\to\infty}{\longrightarrow}0 \}
 $$
 and 
 $$
 F\: \ \begin{cases}
         \mathcal{E} \to & \mathcal{E} \\
         \phi \mapsto & F\phi: x\mapsto \int_x^{\infty}\int_y^{\infty}(1+\phi(z))\nu(z)dzdy.
        \end{cases}
 $$
 From the hypothesis on $\mathcal{E}$ and $\nu,$ $F$ is well defined.
$\mathcal{E}$ is a closed subset of the Banach space $C_0([M,\infty[)$. The choice of $M$ implies that $F$ is a contraction. 
From a classical Banach fixed point argument, there exists a unique positive solution $\phi_1$ in $C([M,+\infty[)$ of $\ref{homogeneous2}$ satisfying
 $\phi(x)\underset{x\to+\infty}{\longrightarrow}0.$

%
Moreover, without loss of generality, we can only consider the case $M=0$.
 
 {\bf Construction of $\phi_2$}: we are looking for a second solution of (\ref{homogeneous}) in the form $\psi_2=\phi_2.\psi_1$.
Integrating the equation we get for $x\geq 0$: 
 \begin{equation*}
  \phi_2(x)=\int_0^x \frac{dy}{(1+\phi_1(y))^2},
 \end{equation*}
and $\psi_2:=\phi_2(1+\phi_1)$ is a second solution of the homogeneous equation (\ref{homogeneous}). Finally, considering that
$\phi_1(x)\to 0$ with $x\to +\infty$, we get the desired estimate for $\phi_2$.
\end{proof}
Of course, we have a similar result for $x\to-\infty.$
This lemma first allows us to give a useful lower bound of $\phi(y;\e)$ at the limit $\e=0.$

\begin{cor}
 \label{minoration}
 Let $\phi=\phi(y;\e)$ be the solution of (\ref{epseq1}). There exists $k>0$ such that, $\forall y\in\R$, 
 $\exists \e_y,\ \forall\e<\e_y$, $\phi(y;\e)\geq k,$ and this uniformly on every compact set in $y$.
\end{cor}

\begin{proof}
 Since $\mu\not\equiv 0$ there exists a nonnegative compactly supported function $\mu_c\not\equiv 0$ such that $0\leq\mu_c \leq \mu.$
 Let us now consider the (unique) solution $\phis=\phis(y;\e)$ of 
 \begin{equation}
\label{epseq2}
\begin{cases}
   -\phis''(y)+(\e+\nu(y))\phis(y)=\mu_c(y) \\
  \phis\in H^1(\R),\ \e>0.
 \end{cases}
\end{equation}
From the first part of Lemma \ref{lemeps}, we know that $\exists K>0,\forall y\in\R,\forall \e> 0$, 
$0<\phis(y;\e)\leq \phi(y;\e)<K.$
Let us recall that for fixed $y\in\R,\ \phis(y;\e)$ is increasing with $\e\to 0$ and bounded by $K$. Hence there exists a positive function
$\phis_0$ such that $\phis(y;\e)\underset{\e\to 0}{\longrightarrow}\phis_0(y)$. Moreover, from the uniform boundedness of $\phis(\e)$ and
Ascoli's theorem, the convergence is uniform for $\phis$ and $\phis'$ in every compact set. Thus, $\phis_0$ satisfies in the classical sense
\begin{equation*}
\begin{cases}
   -\phis_0''(y)+\nu(y)\phis_0(y)=\mu_c(y) \\
  0<\phis_0\leq K.
 \end{cases}
 \end{equation*}
As $\mu_c$ is compactly supported, for $|y|$ large enough, let us say greater than $A>0$, $\phis_0$ is a solution of (\ref{homogeneous}), that is
to say, in the positive semi-axis
\begin{equation*}
 \begin{cases}
 -\phis_0''(y)+\nu(y)\phis_0(y)=0, &  y>A \\
 0<\phis_0(y)\leq K<+\infty & y>A.
 \end{cases}
\end{equation*}
Thus, there exist $\a^+,\b^+$ such that 
$$
\forall y>A,\ \phis_0(y)=\a^+(1+\phi_1(y))+\b^+\phi_2(y)(1+\phi_1(y)),
$$
where $\phi_1$ and $\phi_2$ are defined in Lemma \ref{lemhomo}. Now considering that $\phi_1(y)=o(1)$ and $\phi_2(y)=\varTheta(y)$ in $y\to+\infty$,
as $\phis_0$ is bounded, $\b^+=0.$ Then, as $\phis_0>0,$ $\a^+>0$. We have a similar result for $y<-A,$ with $\b^-=0$ and $\a^->0$.
Finally, define $$k=\frac{1}{2}\min(\a^-,\a^+,\min\{\phis_0(y),y\in[-A,A]\})>0$$ and the proof is concluded.
\end{proof}

\paragraph{Proof of the second part of Lemma \ref{lemeps}}
Differentiating equation (\ref{epseq1}) with respect to $\e$, we get for the derivative $\phi_{\e}$
\begin{equation}
 \label{eqpes2}
-\phi_{\e}''(y;\e)+(\e+\nu(y))\phi_{\e}(y;\e)=-\phi(y;\e).
\end{equation}
Since $\phi$ is positive, we get that $\phi_{\e}$ is negative. Let us denote 
\begin{equation*}
\vp(y)=\vp(y;\e):=-\phi_\e (y;\e)>0. 
\end{equation*}
We have previously seen (in the proof of the first part of Proposition \ref{sourire}) that $\forall y\in\R,$ $\frac{d}{d\e}\vp(y;\e)<0$, \textit{i.e.}
$\vp$ is increasing with $\e\to 0,\e>0.$ Our purpose is to show that in a neighbourhood of $0,$ $\inf(\vp(\e))\underset{\e\to0}{\longrightarrow}+\infty.$
For all $\e>0,$ define the function $\vps=\vps(y;\e)$ as the unique solution of 
\begin{equation}
 \label{soussolderivee}
 \begin{cases}
 -\vps''(y;\e)+(\e+\nu(y))\vps(y;\e)=\min(k,\phi(y;\e)) \\
 \vps\in H^1(\R).
 \end{cases}
\end{equation}
The function $\vps$ is obviously well-defined. By its definition, the elliptic maximum principle ensures us that
$0<\vps\leq \vp,\ \forall y\in\R,\e>0.$ We have also to notice that uniformly on every compact set in $y,$ 
$\min(k,\phi(y;\e))=k$ for $\e$ small enough (consequence of corollary \ref{minoration}).
Assume by way of contradiction that 
\begin{equation}
 \label{hypothesecontr}
 \left(\underset{y\in[-1,1]}{\min}(\vps(y;\e))\right)_\e \textrm{ is bounded.}
\end{equation}
Let us show that it is inconsistent with the fact that $\vps>0,\forall \e>0.$
As $\min(k,\phi(y;\e))$ is uniformly bounded, from Harnack inequalities (see ~\cite{GT}, Theorem 8.17 and 8.18) we know that
for all $R>0,$ there exist $C_1=C_1(R),C_2=C_2(R),$ independent of $\e$, such that for all $\e>0,$ 
$$
\underset{[-R,R]}{\sup}\vps \leq C_1\underset{[-R,R]}{\inf}(\vps+C_2).
$$
Combining this and hypothesis (\ref{hypothesecontr}), we get that $\left(\vps(y;\e) \right)_{\e>0}$ is increasing with $\e\to0$ 
and uniformly in every compact set in $y$. Using the same argument as in the proof of Corollary \ref{minoration}, $(\vps(\e))_\e$
converges locally uniformly to some function $\vps_0$ which satisfies in the classical sense
\begin{equation}
\label{eqlimite2}
\begin{cases}
 -\vps_0''(y)+ \nu(y)\vps_0(y) = k \\
 \vps_0(y)\geq 0,\ \forall y\in\R.
\end{cases}
\end{equation}
So there exist $\a,\b\in\R$ such that $\vps_0=\a(1+\phi_1)+\b\phi_2(1+\phi_1)+\phi_s,$ where $\phi_1,\phi_2$
are defined in Lemma \ref{lemhomo} and $\phi_s$ is a particular solution of (\ref{eqlimite2}). Thus, for $x\geq0,$
$$
\phi_s(y) = -k\left(1+\phi_1(y)\right)\left(1+\phi_1(0)\right)\left( \int_0^y\int_z^y \frac{1+\phi_1(z)}{(1+\phi_1(t))^2}dtdz \right).
$$
Now, recall that $\phi_1>0,\phi_1(y)=o(y)$ as $y$ goes to $+\infty$. So there exists $\gamma>0,$ $\phi_s(y)\underset{y\to\infty}{\sim}-\gamma.y^2.$
As a result, for $y\to \infty$, 
\begin{equation*}
\begin{cases}
 \vps_0(y) = -\gamma.y^2+o(y^2)\underset{y\to+\infty}{\longrightarrow}-\infty \\
 \vps_0\geq0,\ \forall y\in\R,
\end{cases}
\end{equation*}
which is obviously a contradiction. So the first hypothesis (\ref{hypothesecontr}) is false, which gives, combined with 
the monotonicity in $\e,$ 
$$
\underset{y\in[-1,1]}{\min}(\vps(y;\e))\underset{\e\to0}{\longrightarrow}+\infty,
$$
and then, as $\nu$ is continuous and $\nu(0)>0,$ 
$$
\int_\R \phi_\e(y;\e)\nu(y)dy \underset{\e\to0}{\longrightarrow} -\infty,
$$
and the proof of the main Lemma \ref{lemeps} is complete.
\qed

\subsection{Intersection of $\Gamma_1$ and $\Gamma_2$, supersolution}
\paragraph{First case: $D>2d$.}
If $D>2d$, we have of course $\frac{c}{D}<\frac{c}{2d},\ \forall c\geq c_{KPP}.$ Thus,
for $c$ close enough to $c_{KPP}$, $\Gamma_2$ does not intersect the closed convex hull of $\Gamma_1$. But 
since $$\frac{c}{D} \underset{c\to +\infty}{\longrightarrow} +\infty\textrm{ and }\la_2^-(c)\underset{c\to +\infty}{\longrightarrow} 0^+,$$
there exists $$c_*=c_*(D)>c_{KPP}$$ such that $\forall c>c_*,\ \Gamma_1$ and $\Gamma_2$ intersect, and $\forall c<c_*$, $\Gamma_2$
does not intersect the closed convex hull of $\Gamma_1$. Moreover, the strict concavity of $\Gamma_1$ and the strict convexity
of $\Gamma_2$ allow us to assert that for $c=c_*,\ \Gamma_1 \textrm{ and }\Gamma_2$ are tangent on $\la=\la(c_*)$ 
and for $c>c_*,\ c$ close to $c_*$, $\Gamma_1$ and $\Gamma_2$ intersect twice, at $\la(c)^+$ and $\la(c)^-$. 
The different situations are illustrated in fig. (\ref{Dsup2d}).
\\
When $c$ is such that $\la_2^-\leq \frac{c}{D}$, i.e.$\displaystyle c\geq \frac{D}{2\sqrt{dD-d^2}}c_{KPP}$
there is only one solution for $\la=\la(c).$


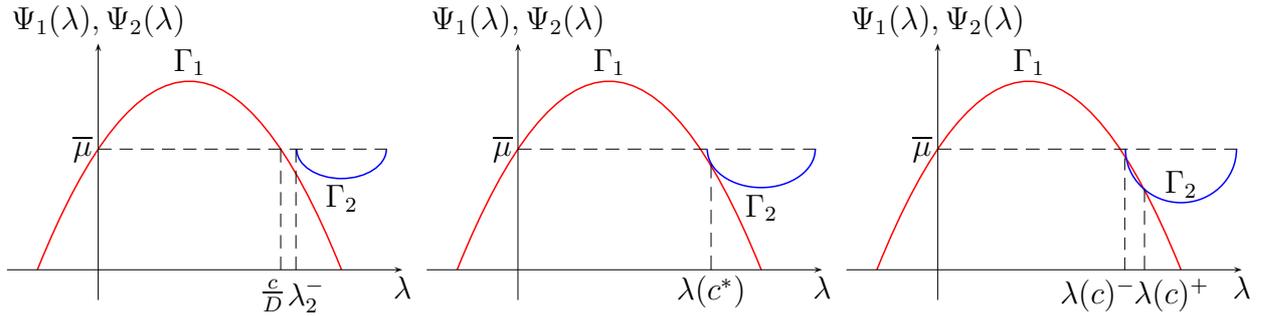
\begin{figure}[ht]
   \begin{minipage}[c]{.31\linewidth}
   \psset{unit=0.4}
\psset{linewidth=0.03}
\begin{pspicture}(-3,-1)(9,10)

\psparabola[linewidth=0.05,linecolor=red](-2,0)(3,6.25)
\psellipticarc[linewidth=0.05,linecolor=blue](8,4)(1.5,1){180}{360}
\psline{->}(-3,0)(10,0)
\psline{->}(0,-1)(0,7.5)
\psline[linestyle=dashed](0,4)(9.5,4)
\uput{0.2}[90](3,6.25){$\Gamma_1$}
\uput{0.2}[-90](8,3){$\Gamma_2$}
\uput{0.2}[-90](10,0){$\lambda$}
\uput{0.2}[180](0,4){$\mub$}
\psline[linestyle=dashed](6,4)(6,0)
\uput{0.3}[-110](6,0){$\frac{c}{D}$}
\psline[linestyle=dashed](6.5,4)(6.5,0)
\uput{0.3}[-70](6.5,0){$\lambda_2^-$}

\uput{0.2}[90](0,7.5){$\Psi_1(\lambda),\Psi_2(\lambda)$}
\end{pspicture}
   \end{minipage} \hfill
   \begin{minipage}[c]{.31\linewidth}
   \psset{unit=0.4}
\psset{linewidth=0.03}

\begin{pspicture}(-3,-1)(9,10)

\psparabola[linewidth=0.05,linecolor=red](-2,0)(3,6.25)
\psellipticarc[linewidth=0.05,linecolor=blue](8,4)(1.8,1.3){180}{360}
\psline{->}(-3,0)(10,0)
\psline{->}(0,-1)(0,7.5)
\psline[linestyle=dashed](0,4)(9.8,4)
\uput{0.2}[90](3,6.25){$\Gamma_1$}
\uput{0.2}[-90](8,2.7){$\Gamma_2$}
\psline[linestyle=dashed](6.35,0)(6.35,3.5)
\uput{0.2}[-90](6.35,0){$\lambda(c^*)$}
\uput{0.2}[-90](10,0){$\lambda$}
\uput{0.2}[180](0,4){$\mub$}
\uput{0.2}[90](0,7.5){$\Psi_1(\lambda),\Psi_2(\lambda)$}

\end{pspicture}
   \end{minipage} \hfill
   \begin{minipage}[c]{.31\linewidth}
   \psset{unit=0.4}
\psset{linewidth=0.03}

\begin{pspicture}(-3,-1)(9,10)

\psparabola[linewidth=0.05,linecolor=red](-2,0)(3,6.25)
\psellipticarc[linewidth=0.05,linecolor=blue](8,4)(1.85,1.8){180}{360}
\psline{->}(-3,0)(10,0)
\psline{->}(0,-1)(0,7.5)
\psline[linestyle=dashed](0,4)(9.85,4)
\psline[linestyle=dashed](6.15,0)(6.15,3.7)
\psline[linestyle=dashed](6.8,0)(6.8,2.7)
\uput{0.3}[-120](6.15,0){$\lambda(c)^-$}
\uput{0.3}[-60](6.8,0){$\lambda(c)^+$}

\uput{0.2}[90](3,6.25){$\Gamma_1$}
\uput{0.2}[90](8,2.2){$\Gamma_2$}

\uput{0.2}[-90](10,0){$\lambda$}
\uput{0.2}[180](0,4){$\mub$}
\uput{0.2}[90](0,7.5){$\Psi_1(\lambda),\Psi_2(\lambda)$}

\end{pspicture}
   \end{minipage}
   \caption{\label{Dsup2d}Case $D>2d$ ; $c<c_*$ (left), $c=c_*$ (middle), $c>c_*$, close to (right)}
\end{figure}

\paragraph{Second case: $D=2d$.}
If $D=2d$, then the point $(\frac{c}{2d},\mub)$ belongs to $\Gamma_1$. Therefore, for all $c>c_{KPP}$, $\Gamma_1$ and
$\Gamma_2$ intersect once at $\la=\la(c)$. We set: 
$$
c_*(2d):=c_{KPP}.
$$
\paragraph{Third case: $D<2d$.}
If $D<2d$, we have $\frac{c}{D}>\frac{c}{2d}$. Then, $\forall c>c_{KPP},\ \la_2^-(c)<\frac{c}{D}$, 
the left part of $\Gamma_2$ is strictly below $\Gamma_1$, and every $c>c_{KPP}$ gives a super-solution. We set again: 
$$
c_*(D):=c_{KPP}.
$$
All of this concludes the proof of Proposition \ref{spreading_speed}.
Moreover, we can assert from geometrical considerations that 
\begin{equation}
\label{inegaliteD}
 \frac{c_*}{D}\leq \frac{c_*-\sqrt{c_*^2-c_{KPP}^2}}{2d}\leq \frac{c_*+\sqrt{c_*^2+4D\mub}}{2D}.
\end{equation}
It was proved in ~\cite{BRR1} that (\ref{inegaliteD}) implies that
\begin{equation*}
 \sqrt{4\mub^2+f'(0)^2}-2\mub\leq \underset{D\to+\infty}{\lim\inf}\frac{c_*^2}{D}\leq 
 \underset{D\to+\infty}{\lim\sup}\frac{c_*^2}{D}\leq f'(0).
\end{equation*}


\subsection{}{Explicit computation of $\Psi_2:=\Psi_2^0$ in the reference case (\ref{BRReq2})}
In the limit case, (\ref{eqgeneralesurphi}) can be written as follows, setting $P(\la)=-d\la^2+c\la-\fp$: 
\begin{equation}\label{eqsurphiBRR}
 -d\phi''(y)+\lp P(\la)+\nub\d_0\rp\phi(y)=\mub\d_0.
\end{equation}
Thus, an explicit computation (see \cite{BRR1} or \cite{Pauthier2}) gives
\begin{equation}\label{solPsi2BRR}
 \Psi_2^0(\la) := \nub\phi(0) = \frac{\nub\mub}{\nub+2\sqrt{dP(\la)}}.
\end{equation}
Notice that this function satisfies all properties given by Proposition \ref{sourire}.

\section{Spreading}
In order to prove that solutions spread at least at speed $c_*$, we are looking for compactly supported
general stationary subsolution in the moving framework at velocity $c<c_*$, arbitrarily close to $c_*.$
We consider the linearised system penalised by $\d>0$ in the moving framework :

\begin{equation}
\label{pendelta}
 \begin{cases}
   \partial_t u-D \partial_{xx} u+c\partial_x u = -\mub u+\int \nu(y)v(t,x,y)dy & x \in \R,\ t>0 \\
 \partial_t v-d\Delta v+c\partial_x v = (f'(0)-\d)v +\mu(y)u(t,x)-\nu(y)v(t,x,y) & (x,y)\in \R^2,\ t>0.
 \end{cases}
\end{equation}
The main result is here the following: 
\begin{prop}\label{soussol}
 Let $c_*=c_*(D)$ be as in the previous section. Then, for $c<c_*$ close enough to $c_*$, there exists $\d>0$
such that (\ref{pendelta}) admits a nonnegative, compactly supported, generalised stationary subsolution
$(\underline{u},\underline{v})\not\equiv (0,0)$.
\end{prop}
 As in the previous section, we will study separately the case $D>2d$, which is the most interesting, and the case $D\leq 2d$.

\subsection{Construction of subsolutions: $D>2d$}
In order to keep the notation as light as possible, we will use the notation $\ft:=f'(0)-\d$ and $\Pt(\la):=-d\la^2+c\la-\ft$, because
all the results will perturb for small $\d>0$. We just have to keep in mind that $\ft<f'(0)$ and $\d\ll1$, hence
$\Pt(\la)>P(\la)$ and $\Pt(\la)-P(\la)\ll 1$.
\\
Our method is to devise a stationary solution of (\ref{pendelta}) not in $\R^2$ anymore, but in the horizontal strip
$\Omega^L = \R \times (-L,L)$, with $L>0$ large enough. Thus, we are solving 
\begin{equation}
 \label{stationnaire}
\begin{cases}
 -DU''+cU'=-\mub U+\int_{(-L,L)}\nu(y)V(x,y)dy & x\in \R \\
 -d\Delta V+c\partial_x V=\ft V+\mu(y)U(x)-\nu(y)V(x,y) & (x,y)\in \Omega^L \\
 V(x,L)=V(x,-L) = 0 & x\in \R.
\end{cases}
\end{equation}
In a similar fashion as in the previous section, we are looking for solutions of the form 
\begin{equation}
\label{soussolexp} 
\begin{pmatrix}
 U(x) \\
V(x,y)
\end{pmatrix}
= e^{\la x} \begin{pmatrix}
                  1 \\ \vp(y)
                 \end{pmatrix},
\end{equation}
where $\vp$ belongs to $H_0^1(-L,L)$. The system on $(\la,\vp)$ reads: 
\begin{equation}
\label{eqsoussol}
\begin{cases}
 -D\la^2+\la c+\mub  = \int_{(-L,L)} \nu(y)\phi(y)dy \\
-d\vp''(y)+(\Pt(\la)+\nu(y))\vp(y)  =  \mu(y) & \vp(-L) = \vp(L)=0.
\end{cases}
\end{equation}
The first equation of (\ref{eqsoussol})  gives a function $\la \mapsto \Psi_1(\la;c)=-D\la^2+\la c+\mub$.
The second equation of (\ref{eqsoussol}) gives a unique solution $\vp=\vp(y;\la,c;L)\in H^1_0(-L,L)$. We associate 
this unique solution with the function $\Psi_2^L(\la;c)=\int_{(-L,L)}\nu(y)\vp(y)dy$.
A solution of the form (\ref{soussolexp}) exists if and only if $\Psi_1(\la;c)=\Psi_2^L(\la;c)$ for some $\la,c$, that is to 
say if and only if $\Gamma_1$ and $\Gamma_2^L$ intersect (with straightforward notations). In this section, the game is to make them
intersect not with real but with complex $\la$.

\paragraph{Study of $\Gamma_1$} 
The function $\la\mapsto \Psi_1$ is exactly the same as in the search for supersolutions. In particular, it does not 
depend in $L$. Thus, the curve $\Gamma_1$ is the same as in the previous section: it is a parabola, symmetric with 
respect to the line $\{\la=\frac{c}{2D}\}$. Notice that being a parabola, its curvature is positive at any point ; it will be important
later.

\paragraph{Study of $\Gamma_2^L$}
The study of $\Gamma_2^L$ is quite similar to that of $\Gamma_2$. It amounts to studying the solutions of
\begin{equation}
\label{eqsoussol2}
\begin{cases}
 -d\vp''(y)+(\Pt(\la)+\nu(y))\vp(y)  =  \mu(y)  & y\in (-L,L)\\
 \vp \in H^1_0(-L,L). 
\end{cases}
\end{equation}
For real $\la$, (\ref{eqsoussol2}) admits solution for $\la \in [\la_{2,\d}^-,\la_{2,\d}^+]$, with
$\la_{2,\d}^\pm = \frac{c\pm\sqrt{c^2-4d\ft}}{2d}$. We may notice that $\la_{2,\d}^-<\la_2^-,\ \la_{2,\d}^+>\la_2^+$, and of course
$\la_{2,\d}^\pm \longrightarrow \la_2^\pm$ as $\d\to 0$. 
With a simple study of (\ref{eqsoussol2}) and using what we proved in proposition (\ref{sourire}), we can assert :
\begin{equation}\label{convergencegamma}
 \underset{L\to\infty}{\lim} \underset{\d\to 0}{\lim} \Psi_2^L(\la;c) = \underset{\d\to 0}{\lim} \underset{L\to\infty}{\lim} \Psi_2^L(\la;c)
 =\Psi(\la;c),
\end{equation}
and this uniformly on every compact set in $]\la^-_2,\la^+_2[\times ]2\sqrt{df'(0)},+\infty[.$

As a consequence, the picture is analogous to the case described in fig. (\ref{Dsup2d}): there exists a unique 
$c_*^L$ (which depends on $\d$) such that $\Gamma_2^L$ intersects $\Gamma_1$ twice if $c>c_*^L,$ close to $c_*^L$, 
once if $c=c_*^L$, and never if $c<c_*^L$ (for real $\la$).\\
 Moreover, since $\Gamma_2^L$ is below $\Gamma_2$, we have $c_{KPP}<c_*^L<c_*$. We also have
$c_*^L\longrightarrow c_*$ as $L\to\infty,\ \d\to 0$.

\paragraph{Complex solutions}
We use the same argument as in ~\cite{BRR1}.
Let us call $\b$ the ordinate of the plane $(\la,\Psi_{1,2}^L)$. For $c=c_*^L$, call $(\la^L_*,\b^L_*)$ the tangent point 
between $\Gamma_1$ et $\Gamma_2^L$. The functions $\Psi_1$ and $\Psi_2^L$ are both analytical in $\la$ at this point, and 
$\frac{d}{d\la}\Psi_1(\la),\quad\frac{d}{d\la}\Psi_2^L(\la)\neq0,$ for $(c,\la)$ in a neighbourhood of $(c_*^L,\la^L_*)$.
Due to the implicit function theorem, there exist $\la_1(c,\b),\ \la_2^L(c,\b)$ defined in a neighbourhood $V_1$ of $(c_*^L,\b^L_*)$,
analytical in $\b$, such that
\begin{equation}
\label{implicite} 
\begin{cases}
 \Psi_1(\la_1(c,\b);c)=\b & \forall (c,\b)\in V_1 \\
 \Psi_2^L(\la_2^L(c,\b);c)=\b & \forall (c,\b)\in V_1.
\end{cases}
\end{equation}

Then, set 
$$
h^L(c,\b)=\la_2^L(c,\b)-\la_1(c,\b),\qquad \textrm{for} \ (c,\b)\in V_1,
$$
and we get: 
\begin{equation}
\label{derivees}
\begin{cases}
\partial_\b h^L(c_*^L,\b_*^L)=0. \\
\partial_{\b\b} h^L(c_*^L,\b_*^L):=2a>0. \\
\partial_c h^L(c_*^L,\b_*^L):=-e<0. 
\end{cases}
\end{equation}
The first point is obvious. The second comes from the fact that $\Gamma_2^L$ is concave and $\Gamma_1$ has a positive curvature at any
point. The third is obvious given the first equation of (\ref{eqsoussol}).


Now, because we are working in a vicinity of $(c_*^L,\b_*^L)$, set :
$$
\xi:=c_*^L-c,\qquad \tau=\b-\b_*^L.
$$
Call $b:=\partial_{c\b}h^L(c_*^L,\b_*^L)$.
From (\ref{implicite}) and (\ref{derivees}), we can assert that there exists a neighbourhood $V_2\subseteq V_1$ of $(c_*^L,\b_*^L)$,
there exists $\eta=\eta(\tau,\xi)$ analytical in $\tau$ in $V_2-(c_*^L,\b_*^L)$, vanishing at $(0,0)$ like $|\tau|^3+\xi^2$, such that
\begin{equation}
 \label{analytical}
(h^L(c,\b)=0,\ (c,\b)\in V_2)\Leftrightarrow (a\tau^2+b\xi\tau+e\xi=\eta(\tau,\xi)).
\end{equation}
Recall that $a$ and $e$ are positive, so the discriminant $\Delta=(b\xi)^2-4ae\xi$ is negative for $\xi>0$ small enough. 
The trinomial $a\tau^2+b\xi\tau+e\xi$ has two roots $\tau_\pm=\frac{-b\xi\pm i\sqrt{4ea\xi-(b\xi)^2}}{2a}$.
Then, from an adaptation of Rouch\'e's theorem (see ~\cite{BRR1}), the right handside of (\ref{analytical}) has two roots,
still called $\tau_\pm$, satisfying $\tau_\pm=\pm i\sqrt{(e/a)\xi}+O(\xi)$.
Reverting to the full notation, we can see that for $c$ strictly less than and close enough to $c_*^L$, there exist
$\b,\la \in \mathbb{C},\ \vp \in H^1_0((-L,L),\mathbb{C})$ satisfying (\ref{eqsoussol}). Since
$\b=\Psi_1(\la)=-D\la^2+c\la+\mub$ and $\b$ has nonzero imaginary part, $\la$ has also nonzero imaginary part.
We can therefore write $(\la,\b)=(\la_1+i\la_2,\b_1+i\b_2)$ and: 
\begin{equation*}
 \begin{pmatrix}
  U \\ V
 \end{pmatrix}
=e^{(\la_1+i\la_2)x}
\begin{pmatrix}
 1 \\ \vp_1(y)+i\vp_2(y)
\end{pmatrix}
\end{equation*}
with
\begin{equation*}
 \begin{cases}
  \la_2,\b_2 \neq 0 \\
  \int\nu(y)\vp_1(y)dy=\b_1=\b_*^L+O(c_*^L-c) \\
  \int\nu(y)\vp_2(y)dy=\b_2=O(\sqrt{c_*^L-c}).
 \end{cases}
\end{equation*}
Thus :
\begin{itemize}
 \item $\Re U >0$ on $(-\frac{\pi}{2\la_2},\frac{\pi}{2\la_2})$ and vanishes at the ends ;
 \item $\Re V>0 \Leftrightarrow \vp_1\cos(\la_2x)>\vp_2\sin(\la_2x)$.
\end{itemize}
The set where $\Re V >0$ is periodic of period $\frac{2\pi}{\la_2}$ in the direction of the road.
Its connected components intersecting the strip $\R\times (-L,L)$ are bounded. The function
$\vp_2$ is continuous in $c$, hence the functions $y\mapsto\vp(y;c)$ are uniformly equicontinuous for $c$ near $c_*^L$.
Since $\nu(0)>0$ and $\int \nu\vp_2=O(\sqrt{c_*^L-c})$, we have $\vp_2(0)=O(\sqrt{c_*^L-c})$, and we can make
one of the connected components of $\{\Re V >0\}$, denoted by $F$, satisfy the property that
$\{(x,0)\in \overline{F}\}$ is arbitrary close to $[-\frac{\pi}{2\la_2},\frac{\pi}{2\la_2}]$.
We can now define the following functions: 
\begin{equation}
 \begin{array}{lcl}
  \underline{u}(x) & := & \begin{cases}
                     \max (\Re{U(x)},0) & \textrm{ if } |x|\leq \frac{\pi}{2\la_2} \\
		     0 & \textrm{ otherwise }
                    \end{cases}
\\
\underline{v}(x,y) & := & \begin{cases}
                     \max (\Re{V(x,y)},0) & \textrm{ if } (x,y)\in \overline{F} \\
		     0 & \textrm{ otherwise }.
                    \end{cases}
 \end{array}
\end{equation}
The choice of $F$ implies that $(\underline{u},\underline{v})$ is a subsolution of (\ref{pendelta}).

\subsection{Subsolution: case $D\leq 2d$}
Now assume that $0\leq D \leq 2d$. In the previous section, we define $c_*(D)=c_{KPP}=2\sqrt{df'(0)}$.
Let $c\leq c_{KPP}$. Thus, $4df'(0)-c^2>0$. Let $\d$ be such that $0<2\d<\frac{4df'(0)-c^2}{4d}=f'(0)-\frac{c^2}{4d}$.
With $\omega=\frac{\sqrt{4d(f'(0)-2\d)-c^2}}{2d}$, we define 
$$
\phi(x)=e^{\frac{c}{2d}x}cos(\omega x)\mathbf{1}_{(-\frac{\pi}{2\omega},\frac{\pi}{2\omega})}.
$$
The function $\phi$ is continuous and satisfies 
$$
-d\phi''+c\phi=(f'(0)-2\d)\phi \qquad \textrm{on} \  (-\frac{\pi}{2\omega},\frac{\pi}{2\omega}).
$$
Then, let us choose $R>0$ such that the first eigenvalue of $-\partial_{yy}$ in $(-R,R)$ is equal to $\frac{\d}{d}-\a$, and 
$\psi_R$ an associated nonnegative eigenfunction in $H_0^1(-R,R)$, where $0<\a<\d$. The function $\psi_R$ satisfies 
$$
-d\psi_R''=(\d-\a)\psi_R \ \textrm{in}\ (-R,R),\ \psi_R(y)>0,\ \forall |y|<R,\ \psi_R(R)=\psi_R(-R)=0.
$$
We extend $\psi_R$ by 0 outside $(-R,R)$. Let $M>0$ such that $\forall |y|>M-R,\ \nu(y)\leq \a,$ which is possible since $\nu(y)\to0$
with $y\to\pm\infty.$
The function 
$$
\underline{V}(x,y):=\phi(x)\psi_R(|y|-M)
$$
is a solution of
$$
\begin{cases}
-d\Delta V+c\partial_x V=(f'(0)-\d)V-\a V \\
x\in (-\frac{\pi}{2\omega},\frac{\pi}{2\omega}), \ |y|\in(M-R,M+R),
\end{cases}
$$
vanishing on the boundary. Hence, from the choice of $M$ and $\a,$ $(0,\underline{V})$ is a nonnegative compactly supported
subsolution of (\ref{pendelta}), non identically equal to $(0,0)$ ; which concludes the proof of Proposition \ref{soussol}. The proof 
of the main Theorem \ref{spreadingthm} follows as in ~\cite{BRR1}.

\section{The intermediate model (\ref{RPSL})}

\paragraph{Formal derivation of the semi-limit model}
Starting from the full model (\ref{RPeq}), we consider normal (i.e. integral) exchange from the field to the road but localised exchange from 
the road to the field. Formally, we define $\mu_\e = \frac{1}{\e}\mu(\frac{y}{\e})$ and take the limit with $\e\to 0$ of the system (\ref{RPepsSL}) :
\begin{equation}
 \label{RPepsSL}
\begin{cases}
 \partial_t u-D \partial_{xx} u = -\mub u+\int \nu(y)v(t,x,y)dy & x \in \R,\ t>0 \\
 \partial_t v-d\Delta v = f(v) +\mu_\e(y)u(t,x)-\nu(y)v(t,x,y) & (x,y)\in \R^2,\ t>0.
\end{cases}
\end{equation}
There is no influence in the first equation (the dynamic on the road), which is the same in the limit system. Though
the second equation in (\ref{RPepsSL}) tends to
\begin{equation*}
 \partial_t v-d\Delta v = f(v)-\nu(y)v(t,x,y),\qquad (x,y)\in \R\times\R\diagdown\{0\},\ t>0.
\end{equation*}
It remains to determine the limit condition between at the road. We may assume that for $\e=0$ $v$ is still continuous
at $y=0$. Now set $\xi=y/\e$ and $\vt(t,x,\xi):=v(t,x,y)$. The second equation in (\ref{RPepsSL}) becomes in the $(t,x,\xi)$-variables
$$
\e^2\left(\partial_t \vt-d\partial_{xx}\vt-f(\vt)+\nu(\xi)\vt(t,x,\xi)\right)-d\partial_{\xi\xi}\vt = \e\mu(\xi)u(t,x).
$$
Passing to the limit, it yields, in the $y$-variable: 
$$
-d\left(\partial_y v(t,x,0^+)-\partial_y v(t,x,0^-) \right) = \mub u(t,x).
$$
Consequently, the formal limit system of (\ref{RPepsSL}) should be (\ref{RPSL}) presented in the Introduction,
which is the system we will study from now. Our assumptions on $\nu$ and $f$ are the same as above.
The investigation is similar to the one done for the model (\ref{RPeq}), and we will only
develop the parts which differ.

\paragraph{Comparison principle}
Throughout this section, we will call a supersolution of (\ref{RPSL}) a couple $(\us,\vs)$ satisfying, in the classical sense, the following system: 
\begin{equation}
\label{RPSLsuper}
 \begin{cases}
  \partial_t \us-D \partial_{xx} \us \geq v(x,0,t)-\mub u +\nu(y)v(t,x,y) & x\in\R,\ t>0\\
\partial_t v-d\Delta v\geq f(v) -\nu(y)v(t,x,y) & (x,y)\in\R\times\R^*,\ t>0\\
v(t,x,0^+)=v(t,x,0^-), & x\in\R,\ t>0 \\
-d\left\{ \partial_y v(t,x,0^+)-\partial_y v(t,x,0^-) \right\}\geq \mub u(t,x) & x\in\R,\ t>0,
 \end{cases}
\end{equation}
which is also continuous up to time $0.$
Similarly, we will call a subsolution of (\ref{RPSL}) a couple $(\su,\sv)$ satisfying (\ref{RPSLsuper}) with the inverse inequalities 
(\textit{i.e.} the $\geq$ signs replaced by $\leq$. We now need a comparison principle in order to get monotonicity for solutions :

\begin{prop}
 \label{comparaisonSL}
Let $(\su,\sv)$ and $(\us,\vs)$ be respectively a subsolution bounded from above and a supersolution
bounded from below of (\ref{RPSL}) satisfying $\su\leq\us$ and $\sv\leq\vs$ at $t=0$. Then, either $\su<\us$ and
$\sv<\vs$ for all $t>0$, or there exists $T>0$ such that $(\su,\sv)=(\us,\vs),\ \forall t\leq T.$
\end{prop}

We omit the proof.

\paragraph{Long time behaviour and stationary solutions}
We want to show that any (nonnegative) solution of (\ref{RPSL}) converges locally uniformly to
a unique stationary solution $(U_s,V_s)$, which is bounded, positive, $x$-independent, and of course is solution of the 
stationary system of equations (\ref{stateqSL}): 
\begin{equation}
\label{stateqSL}
\begin{cases}
-D U''(x)  = -\mub U +\int\nu(y)V(x,y) \\
d\Delta V(x,y)  = f(V) -\nu(y)V(x,y) \\
V(x,0^+) = V(x,0^-) \\
-d\left\{ \partial_y V(x,0^+)-\partial_y V(x,0^-) \right\}=\mub U(x).
\end{cases}
\end{equation}
Proofs of Propositions \ref{stationary1} and \ref{stationary2} can be easily adapted to this new system. The only nontrivial point lies
in the existence of an $L^\infty$ a priori estimate. 
Set $\la=\frac{\nub}{d}$. From conditions on the reaction term, there
 exists $M_1$ such that $\forall s>M_1,\ f(s)<-\frac{\nub^2}{d}s.$ Now, set 
 $$
 M=\max (M_1,\frac{\nub}{\mub}\|u_0\|_\infty,\|v_0\|_\infty)
 $$
 and the couple $(\overline{U},\overline{V})$ given by
 $$
 \overline{V}(y)=M(1+e^{-\la|y|}),\ \overline{U}=\frac{1}{\mub}\int_\R \nu(y)\overline{V}(y)dy
 $$
 is a supersolution of (\ref{stateqSL}) which is above $(u_0,v_0)$.
 
 The proof of the corresponding Proposition \ref{liouville} follows easily.

\paragraph{Exponential solutions, spreading}
We are looking for solutions of the linearised system: 
\begin{equation}
\label{RPSLli}
\begin{cases}
\partial_t u-D \partial_{xx} u= v(x,0,t)-\mub u +\nu(y)v(t,x,y) & x\in\R,\ t>0\\
\partial_t v-d\Delta v=f'(0)v -\nu(y)v(t,x,y) & (x,y)\in\R\times\R^*,\ t>0\\
v(t,x,0^+)=v(t,x,0^-), & x\in\R,\ t>0 \\
-d\left\{ \partial_y v(t,x,0^+)-\partial_y v(t,x,0^-) \right\}=\mub u(t,x) & x\in\R,\ t>0,
\end{cases}
\end{equation}

and these solutions will be looked for under the form
\begin{equation}
\label{solexpSL} 
\begin{pmatrix}
 u(t,x) \\
v_1(t,x,y) \\
v_2(t,x,y)
\end{pmatrix}
= e^{-\la(x-ct)} \begin{pmatrix}
                  1 \\ \phi_1(y) \\ \phi_2(y)
                 \end{pmatrix}
\end{equation}
where $\la,c$ are positive constants and $\phi$ is a nonnegative function in $H^1(\R)$,
with $v=v_1,\ \phi=\phi_1$ for $y\geq 0$ and $v=v_2, \phi=\phi_2$ for $y\leq0$. The system in $(\la,\phi)$ reads 
\begin{equation}
\label{laphi}
\begin{cases}
 -D\la^2+\la c+\mub  = \int \nu(y)\phi(y)dy \\
-d\phi_1''(y)+(\la c-d\la^2-f'(0)+\nu(y))\phi_1(y)  =  0 & y\geq 0. \\
-d\phi_2''(y)+(\la c-d\la^2-f'(0)+\nu(y))\phi_2(y)  =  0 & y\leq 0. \\
\phi_1(0)=\phi_2(0) & i.e. \ \phi \textrm{ is continuous.} \\
-\phi'_1(0)+\phi'_2(0) = \frac{\mub}{d}.
\end{cases}
\end{equation}

The study is exactly the same as in the third section. The only point which deserves some explanation is 
the well-posedness of (\ref{laphi}).
For $M>0$ let us consider $\vp_M$ the unique solution of
\begin{equation}
 \label{eqM}
 \begin{cases}
  -d\vp_M''(y)+(P(\la)+\nu(y))\vp_M(y) = 0 & y\in]0,+\infty[ \\
  \vp_M(0)=M & \vp_M\in H^1(\R^+).
 \end{cases}
\end{equation}
Let us show the following lemma, which will prove the well-posedness of (\ref{laphi}):
\begin{lem}
\label{wellposed}
 \begin{enumerate}
 \item $M\mapsto \vp_M'(0)$ is decreasing ;
 \item $ \vp_M'(0) \underset{M\to 0}{\longrightarrow}0$ ;
 \item $ \vp_M'(0) \underset{M\to +\infty}{\longrightarrow}-\infty$.
 \end{enumerate}
\end{lem}

\begin{proof}
 Let us consider $M_1,M_2$ with $0<M_1<M_2$, $\vp_{M_1},\vp_{M_2}$ the associated solutions of (\ref{eqM}). The elliptic maximum principle 
 yields $0<\vp_{M_1}(y)<\vp_{M_2}(y),\ \forall y\geq0$ and Hopf's lemma gives $0>\vp_{M_1}'(0)>\vp_{M_2}'(0),$ which proves the first point.
 
 Then, if we integrate (\ref{eqM}) we get 
 $$
 \vp_M'(0)=-\frac{1}{d}\int_0^\infty (P(\la)+\nu(y))\vp_M(y)dy.
 $$
 Let us now consider $\vps_M$ the (unique) solution of 
 \begin{equation*}
   \begin{cases}
  -d\vps_M''(y)+P(\la)\vps_M(y) = 0 & y\in]0,+\infty[ \\
  \vps_M(0)=M & \vps_M\in H^1(\R^+).
 \end{cases}
 \end{equation*}
$\vps_M$ is a supersolution of (\ref{eqM}). Thus, $\vps_M(y)\geq\vp_M(y),\ \forall y\geq0.$ Moreover we have an explicit expression for $\vps_M$: 
$\vps_M(y)=M\exp(-\sqrt{\frac{P(\la)}{d}}y)$. Hence,
$$
0\leq-\vp_M'(0)\leq\frac{M}{d}\int_0^\infty(P(\la)+\nu(y))e^{-\sqrt{\frac{P(\la)}{d}}y}dy
$$
and 
$$
-\vp_M'(0) \underset{M\to0}{\longrightarrow}0
$$
uniformly in $\la$, which proves the second point.

In the same way, the unique solution $\svp$ of
 \begin{equation*}
   \begin{cases}
  -d\svp_M''(y)+(P(\la)+\|\nu\|_\infty)\svp_M(y) = 0 & y\in]0,+\infty[ \\
  \svp_M(0)=M & \svp_M\in H^1(\R^+).
 \end{cases}
 \end{equation*}
is a subsolution of (\ref{eqM}), and $\svp_M(y)\leq\vp_M(y),\ \forall y\geq0.$ Hence,

\begin{equation*}
 \begin{array}{lcl}
-\vp_M'(0) & \geq & \frac{1}{d}\int_0^\infty (P(\la)+\nu(y))\svp_M(y)dy \\
 & \geq & \frac{M}{d}\int_0^\infty (P(\la)+\nu(y))e^{-\sqrt{\frac{P(\la)+\|\nu\|_\infty}{d}}y}du \\
-\vp_M'(0) & \to & +\infty \textrm{ as }M\to +\infty,
 \end{array}
\end{equation*}
which concludes the proof of Lemma \ref{wellposed}.
 \end{proof}
 
 The corresponding Proposition \ref{spreading_speed} and Theorem \ref{spreadingthm} follows as in the previous part.

 \section{The large diffusion limit $D\to+\infty$}
The behaviour of the spreading speed $c^*$ as $D$ goes to $+\infty$ has already been investigated in \cite{BRR1}
for the initial model (\ref{BRReq2}). It has been shown that there exists $c_\infty>0$ such that 
$$
\frac{c^*(D)}{\sqrt{D}}\underset{D\to+\infty}{\longrightarrow}c_\infty.
$$
In the following Proposition, we show the robustness of this result and extend it to the general cases (\ref{BRReq2})-(\ref{RPSL2}).
We also give an asymptotic behaviour as $\fp$ tends to $+\infty.$

\begin{prop}\label{asymptoticDfp}
Let us consider any of the systems (\ref{BRReq2})-(\ref{RPSL2}) with fixed parameters $d,\nub,\mub.$ Let $c^*(D,\fp)$
be the associated spreading speed given by Theorem \ref{spreadingthm}.
\begin{enumerate}
 \item There exists $c_\infty,$ $\displaystyle \frac{c^*(D,\fp)}{\sqrt{D}}\underset{D\to+\infty}{\longrightarrow}c_\infty.$
 \item $c_\infty$ satisfies $\displaystyle c_\infty\underset{\fp\to+\infty}{\sim}\sqrt{\fp}.$
\end{enumerate}
 That is, with $D\to+\infty$ and $\fp\to+\infty,$ we have $c_0^*\sim\sqrt{\fp D},$ \textit{i. e.} half of the KPP spreading speed 
 for a reaction-diffusion on the road.
\end{prop}

\paragraph{Proof of Proposition \ref{asymptoticDfp}}
We prove the result for the nonlocal system (\ref{RPeq}), the other cases being similar. We set
$$
\ut(t,x)=u(t,\sqrt{D}x),\quad \vt(t,x,y) = v(t,\sqrt{D}x,y).
$$
The system in the rescaled variables becomes
\begin{equation}\label{systemDinfty}
 \begin{cases}
 \partial_t\ut-\partial_{xx}\ut = -\mub\ut+\int\nu(y)\vt(t,x,y)dy \\
 \partial_t\vt-d\lp\partial_{yy}\vt+\frac{1}{D}\partial_{xx}\vt\rp = f(\vt)+\mu(y)\ut(t,x)-\nu(y)\vt(t,x,y).
 \end{cases}
\end{equation}
The $(c,\la,\phi)-$system associated to (\ref{systemDinfty}) is then
\begin{equation*}
 \begin{cases}
\la c-\la^2+\mub = \int \nu\phi \\
-d\phi''(y)+\lp\la c-\frac{d}{D}\la^2-\fp+\nu(y)\rp\phi(y) = \mu(y).
 \end{cases}
\end{equation*}
Hence we get that $c^* = \sqrt{D}\ct$ where $\ct$ is the first $c$ such that the graphs of $\Psit_1$ and $\Psit_2$
intersect, where $\Psit_1$ and $\Psit_2$ are defined as follows:
$$
\Psit_1 : \la \longmapsto \la c -\la^2+\mub
$$
and
$$
\Psit_2 : \begin{cases}
           ]\tilde{\la}^-,\tilde{\la}^+[ & \longrightarrow \R \\
           \la & \longmapsto \int\nu\phi
          \end{cases}
$$
where $\phi$ is the unique $H^1$ solution of 
\begin{equation}\label{eqphirescalee}
 -d\phi''(y)+\lp\la c-\frac{d}{D}\la^2-\fp+\nu(y)\rp\phi(y) = \mu(y)
\end{equation}
and $\displaystyle\tilde{\la}^\pm = \frac{D}{2d}\lp c\pm\sqrt{c^2-4\frac{d\fp}{D}}\rp.$ We can see that
as $D$ tends to $+\infty,$ $\displaystyle\tilde{\la}^-=\frac{\fp}{c}+o(1)$ and $\displaystyle\tilde{\la}^+\to+\infty.$
Behaviours of $\Psit_{1,2}$ have already
been studied above. $\Psit_1$ is a concave parabola, $\Psit_2$ is strictly convex, symmetric with respect to 
$\{\la=\frac{cD}{2d}\}.$ Moreover, it has been showed that the solution $\phi$ of (\ref{eqphirescalee}) is
 bounded in $L^\infty,$ uniformly in $\la,c,D.$ It is also pointwise strictly decreasing in $\lp\la c-\frac{d}{D}\la^2-\fp\rp.$ 
Now, let $\vp$ be the $H^1$ solution of the limit system defined for $\la>\frac{\fp}{c}$
\begin{equation}\label{eqphiDinfty}
 -d\vp''(y)+\lp\la c-\fp+\nu(y)\rp\vp(y) = \mu(y).
\end{equation}
From the maximum principle and the monotonicity of $\phi$ with respect to the nonlinear eigenvalue, we can easily see 
that $\displaystyle \lV\vp-\phi\rV_{L^\infty}\to 0$ as $D\to\infty,$ locally uniformly in $\la,c.$ Hence, $\Psit_2$ tends to  
$\Psit_{2,\infty}$ defined by
$$
\Psit_2 : \begin{cases}
           ]\frac{\fp}{c},+\infty[ & \longrightarrow \R \\
           \la & \longmapsto \int\nu\vp
          \end{cases}
$$
where $\vp$ is the unique solution of (\ref{eqphiDinfty}), and $\ct$ tends to $c_\infty,$ where 
$c_\infty$ is the first $c$ such that the graphs of $\Psit_1$ and $\Psit_{2,\infty}$ intersect.
This concludes the proof of the first part of Propostion \ref{asymptoticDfp}.

\begin{figure}[!ht]
\centering
\begin{tikzpicture}[scale=4]
 \draw [->] (-0.1,0) -- (2.5,0);
 \draw [->] (0,-0.1) -- (0,1.4);
 \draw [blue,very thick,domain=0:1.48] plot (\x,{(0.8*\x-\x^2+1)});
 \draw [red,very thick,domain=1.25:2.4] plot(\x,{1/(1+2*sqrt(0.8*\x-1))});
 \draw (0.5,1.3) node {$\Psit_1$};
 \draw (2,0.2) node {$\Psit_{2,\infty}$};
 \draw (1.48,0) node[below right] {$\frac{c+\sqrt{c^2+4\mub}}{2}$};
 \draw [dashed] (0.8,0) -- (0.8,1);
 \draw (0.8,0) node[below] {$c$};
 \draw [dashed] (0,1) -- (1.25,1);
 \draw [dashed] (1.25,1) -- (1.25,0);
 \draw (1.25,0) node[below] {$\frac{\fp}{c}$};
 \draw (0,1) node[left] {$\mub$};
 \end{tikzpicture}
 \caption{\label{FigDinfini}Curves in the limit case $D\to\infty$}
\end{figure}
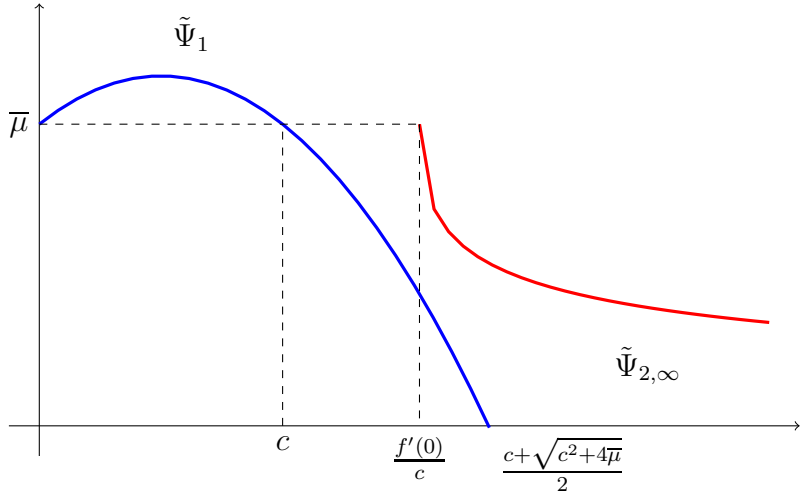

For the second part of Proposition \ref{asymptoticDfp}, we can see from geometric considerations (see figure \ref{FigDinfini})
that $c_\infty$ must satisfy 
\begin{equation}\label{ineqcinfty}
c_\infty \leq \frac{\fp}{c_\infty} \leq \frac{c_\infty+\sqrt{c_\infty^2+4\mub}}{2}.
\end{equation}
Passing to the limit $\fp\to+\infty$ in (\ref{ineqcinfty}) yields the expected result.
\qed

\section{Enhancement of the spreading speed in the semi-limit case (\ref{RPSL2})}
This section is devoted to the semi-limit model (\ref{RPSL2}) and the proof of Proposition \ref{vitessemaxRPSL2}.
For $\mub>0,$ let 
$$
\Lambda_{\mub} = \{\mu\in C_0(\R),\mu\geq 0, \int\mu=\mub,\mu \textrm{ is even} \}.
$$
Now, for fixed constants $d,D,\nub,\fp,$ for any function $\mu\in\Lambda_{\mub},$ let
$c^*(\mu)$ be the spreading speed associated to the semi-limit system (\ref{RPSL2}) with 
exchange function from the road to the field $\mu.$ 
 Let $c^*_0$ the spreading speed associated with the limit system (\ref{BRReq2}) with the same parameters and exchange rate
 from the road to the field $\mub.$

\subparagraph{Proof of Proposition \ref{vitessemaxRPSL2}}
If $D\leq 2d,$ then for all systems, $c^*=2\sqrt{d\fp}=c_K$ and the result is obvious. We consider only the case $D>2d.$
Let $c>2\sqrt{d\fp},$ $\displaystyle\la_2^\pm=\frac{c\pm\sqrt{c^2-c_K^2}}{2d}.$ Then, for all $\la\in]\la_2^-,\la_2^+[,$ 
the $(c,\la,\phi)-$equation (\ref{eqgeneralesurphi}) associated to the semi-limit system (\ref{RPSL2}) can be written
as follows:
\begin{equation}\label{eqphiRPSL2}
 \begin{cases}
 -d\phi''(y)+\lp\la c-d\la^2+\fp\rp\phi(y) = \mu(y) \qquad y>0\\
 -2d\phi'(0)=-\nub\phi(0), \qquad \phi\in H^1(\R^+).
 \end{cases}
\end{equation}
We keep in mind that we are interested in the behaviour of 
$$
\Psi_2(\la;\mu):=\nub\phi(0)
$$
where $\phi$ is the unique solution of (\ref{eqphiRPSL2}). For the sake of simplicity, we set
$$
P(\la)=\la c-d\la^2-\fp,\qquad \a^2=\frac{P(\la)}{d}.
$$ 
From the variation of constants method and the boundary conditions in 0 and $+\infty$ we have
\begin{equation*}
 \phi(y) =  e^{\a y}\lp K_1-\frac{1}{2\a}\int_0^y e^{-\a z}\frac{\mu(z)}{d}dz\rp + 
   e^{-\a y}\lp K_2+\frac{1}{2\a}\int_0^y e^{\a z}\frac{\mu(z)}{d}dz\rp.
\end{equation*}
where
$$
K_1=\frac{1}{2\a}\int_0^\infty e^{-\a z}\frac{\mu(z)}{d}dz,\ K_2 = \frac{2\a d-\nub}{2\a(2\a d+\nub)}\int_0^\infty e^{-\a z}\frac{\mu(z)}{d}dz.
$$
We finally get, returning in the $(\la,c)-$variables,
\begin{equation}\label{formulePsi2RPSL2}
 \Psi_2(\la;\mu):=\nub\phi(0)=\frac{2\nub}{\nub+2\sqrt{dP(\la)}}\int_0^\infty e^{-\sqrt{\frac{P(\la)}{d}}z}\mu(z)dz.
\end{equation}
Now, since $e^{-\sqrt{\frac{P(\la)}{d}}z}\leq1$ for all $z\geq0$ and $\mu$ being nonnegative and even, it is 
easily seen that
\begin{equation}\label{ineqRPSL2}
 \Psi_2(\la;\mu)\leq \Psi_2^0(\la;\mub)
\end{equation}
where $\Psi_2^0$ is given by the limit model (\ref{BRReq2}) associated to the same constants and exchange term $\mub.$
Hence, the above inequality (\ref{ineqRPSL2}) allows us to assert that
$$
\forall \mu\in\Lambda_{\mub},\ c^*(\mu)\leq c^*_0.
$$
Then, stating $c=c^*_0$, let us consider any approximation to the  identity sequence in (\ref{formulePsi2RPSL2}). 
For any $\mu\in\Lambda_{\mub},\e>0,$ set $\mu_\e(y)=\frac{1}{\e}\mu\lp\frac{y}{\e}\rp$. Then we get that 
$\Psi_2(\la;\mu_\e)$ converges to $\Psi_2^0(\la;\mub)$ as $\e$ goes to 0, uniformly in any compact set in $]\la_2^-,\la_2^+[$
 in $\la.$ Hence, 
 $$
 c^*(\mu_\e)\underset{\e\to 0}{\longrightarrow}c_0^*
 $$
 and the proof of Proposition \ref{vitessemaxRPSL2} is concluded.
 \qed

 \section{Self-similar exchanges for the semi-limit case (\ref{RPSL})}
 Considering the above result, it may seem natural that in the opposite case (\ref{RPSL}), that is
when exchanges from the road to the field are localised on the road, the spreading speed would also be maximum
for localised exchange from the field to the road.
In order to compare the spreading speed associated to the initial model (\ref{BRReq2}) and 
the one given by an integral model (\ref{RPSL}), it is first natural to look for the 
behaviour of the spreading speed when replacing the exchange function $\nu$ by a
self-similar approximation of a Dirac mass $\frac{1}{\e}\nu(\frac{y}{\e}).$
Hence, for a fixed constant rate $\nub,$ we will consider an exchange function of the form 
\begin{equation*}
 \nu\in\Lambda_{\nub}:=\{\nu\in C_0(\R),\nu\geq 0, \int\nu=\nub,\nu \textrm{ is even} \}.
\end{equation*}
For fixed constant $\fp,d,D,\mub,$ and $\nu\in\Lambda_{\nub}$ let $c^*_0$ be the spreading speed associated
to the limit system (\ref{BRReq2}), and $c^*(\e)$ the spreading speed associated to the semilimit model (\ref{RPSL}) with
exchange term
\begin{equation*}
 \nu_\e : y\longmapsto \frac{1}{\e}\nu\lp\frac{y}{\e}\rp.
\end{equation*}
The $(c,\la,\phi)-$equation (\ref{eqgeneralesurphi}) associated is
\begin{equation}\label{eqsurphiRPSLlim}
\begin{cases}
 -d\phi''(y)+(P(\la)+\nu_\e(y))\phi(y)  =  \mub\d_0 \\
 \phi \in H^1(\R),\ \phi \textrm{ is continuous}. 
\end{cases}
\end{equation}
The $\Psi_2$ function is given by
$$
\Psi_2(\la,c;\e)=\int_\R \nu_\e(y)\phi(y)dy
$$
where $\phi$ is the unique solution of (\ref{eqsurphiRPSLlim}) and $P(\la)=\la c-d\la^2-\fp.$
An integration of (\ref{eqsurphiRPSLlim}) yields
the following expression for $\Psi_2$
\begin{equation}\label{eqPsi2}
 \Psi_2(\la,c;\e)=\mub-P(\la)\int_\R \phi(y;\la,c,\e)dy
\end{equation}
from which we get the next proposition.

\begin{prop}\label{deriveePsi2positive}
 The function $\Psi_2,$ defined by (\ref{eqPsi2}) and (\ref{eqsurphiRPSLlim}), is continuously differentiable
 in all variables $\la,c,\e$
 up to $\e=0$ and satisfies for all $\la,c$
 $$
 \frac{d}{d\e}{\Psi_2}_{|\e=0}>0.
 $$
\end{prop}

Considering the monotonicity of $\Psi_2$ with respect to $c,$ this provides the corollary
\begin{cor}\label{corollaireVitesse}
Let us consider $c^*$ as a function of the $\e$ variable. Then there exists $\e_0,$ 
$$
\forall \e<\e_0,\ c^*(\e)>c^*_0
$$
In other words, the Dirac mass is a local minimizer for the spreading speed when considering approximation of Dirac functions.
\end{cor}

\paragraph*{Proof of Proposition \ref{deriveePsi2positive}}
Throughout the proof, the function $\phi,$ depending on $\e,$ will be the solution of (\ref{eqsurphiRPSLlim}), 
and we will denote 
$$
\vp:=\frac{d}{d\e}\phi
$$
its derivative with respect to $\e.$ Moreover, once again, we set $d=1$
for the sake of simplicity, and consider an exchange function $\nu$ with support in
$[-1,1].$ Differentiating (\ref{eqsurphiRPSLlim}) we obtaint that $\vp$ is the unique solution of
\begin{equation}\label{deriveephiSL}
 \begin{cases}
 -\vp''(y)+\lp P(\la+\frac{1}{\e}\nu\lp\frac{y}{\e}\rp)\rp\vp(y) = \frac{1}{\e^2}g\lp\frac{y}{\e}\rp\phi(y) \\
 \vp\in H^1(\R)
 \end{cases}
\end{equation}
where the function $g,$ with compact support in $[-1,1],$ is defined by 
\begin{equation}\label{defg}
 g:z\mapsto \frac{d}{dz}[z.\nu(z)].
\end{equation}
  
 Thanks to (\ref{eqPsi2}), it is enough to prove that $\vp$ tends to a negative function as $\e$ goes to 0, uniformly $L^1$ near
 $\e=0.$ The proof is divided in four steps. We first recall the convergence of $\phi$ as $\e$ goes to 0.
 Then, the most important step is the convergence of the righthandside of (\ref{deriveephiSL}) to
 a Dirac measure of negative mass. The third step is to find some uniform boundedness for the sequence $\lp\vp\rp_\e,$ 
 in order to finally pass to the limit and conclude the proof.
 
\subparagraph*{Convergence of $\phi$} It has been proved in \cite{Pauthier2}
that $\phi$ converges in the $C^1$ norm to 
\begin{equation}\label{phi0}
 \phi_0:y\mapsto \frac{\mub}{\nub+2\sqrt{P(\la)}}e^{-\sqrt{P(\la)}|y|}
\end{equation}
as $\e$ goes to 0, and this convergence is locally uniform in $\la,c.$
Actually, the monotonicity of $\phi$ with respect to $y$ makes the proof easier.

\subparagraph*{Convergence of the righthandside of (\ref{deriveephiSL}) to a Dirac mass}
As its support shrinks to 0, and thanks to the reguarity of $g$ and $\phi$ uniformly in $\e,$ 
it is enough to prove the convergence of the mass to get the convergence in the sense of distribution. Let
us consider the integral
\begin{equation*}
 I(\e):=\int_\R \frac{1}{\e^2}g\lp\frac{y}{\e}\rp\phi(y;\e)dy.
\end{equation*}
Evenness of $g$ and $\phi,$ compact support of $g,$ and a Taylor formula yield
\begin{equation}\label{int1}
 \frac{1}{2}I(\e)=\frac{1}{\e}\int_0^1 g(z)\lp\phi(0)+\e z\phi'(0)+\int_0^{\e z}(\e z-t)\phi''(t)dt\rp dz.
\end{equation}
Recall that $g$ is defined by (\ref{defg}), so $\displaystyle \int_0^1 g(z)dz=0.$ An integration by parts
gives 
$$
\int_0^1 z.g(z)dz=-\int_0^1z.\nu(z)dz>0.
$$
It remains to determine the last term in (\ref{int1}) given by
\begin{equation*}
 I_2=\frac{1}{\e}\int_0^1 g(z)\int_0^{\e z}(\e z-t)\phi''(t)dtdz.
\end{equation*}
Recall that $\phi$ is a solution of (\ref{eqsurphiRPSLlim}) and we get locally in $\la,c$
\begin{align}
 I_2 & = \e\int_0^1 g(z)\int_0^z \lp z-u\rp\lp P(\la)+\frac{1}{\e}\nu(u)\rp\phi(\e u)dudz \nonumber \\
  & = \int_0^1 g(z)\int_0^z \lp z-u\rp\nu(u)\phi(\e u)dudz +O(\e). \label{intC}
\end{align}
With the uniform boundedness of $\lV\phi'\rV_{L^\infty}$ in $\e,$ (\ref{intC}) becomes
\begin{align}
 I_2 & = \phi(0)\lp \int_0^1 z.g(z)\int_0^z \nu(u)dudz -  \int_0^1 g(z)\int_0^z u.\nu(u)dudz \rp + O(\e)\nonumber \\
  & = -\phi(0)\int_0^1 z\nu(z)\int_0^z \nu(u)dudz+ O(\e). \label{intC12}
\end{align}
Inserting (\ref{intC12}) in (\ref{int1}), and with the convergence together to its derivative of $\phi$ to $\phi_0$ defined by (\ref{phi0}), 
we get, as $\e$ tends to 0,
\begin{equation}\label{integralelimite}
I \underset{\e\to 0}{\longrightarrow}\mub\int_0^1\lp \frac{1}{\nub+2\sqrt{P(\la)}}\int_{-z}^z\nu(u)du - 1 \rp z.\nu(z)dz:=I_0.
\end{equation}
We notice that $I_0<0.$

\subparagraph*{Uniform boundedness for $\lV\vp\rV_{L^\infty}$}Once again, let us set $\a^2:=P(\la).$
As $g$ is compactly supported and even, for all $\e>0,$ there exists $K(\e)$ such that
\begin{equation}\label{yplusgdqueeps}
 \forall |y|>\e,\qquad \vp(y)=K(\e) e^{-\a|y|}.
\end{equation}
We do the change of variable $\xi=\frac{y}{\e}$ and use the same notation $\vp(\xi)=\vp(y)$ for the sake of simplicity.
Equation (\ref{eqsurphiRPSLlim}) becomes in the $\xi-$variable
\begin{equation}\label{eqenxi}
\begin{cases}
 -\vp''(\xi)+\lp \e^2\a^2+\e\nu(\xi) \rp\vp(\xi)=g(\xi)\phi(\e\xi) \\
 \vp(\pm1)=K(\e)e^{-\a\e}.
\end{cases}
\end{equation}
As for Theorem \ref{thmdenonacceleration}, let us set 
\begin{equation*}
\vp(\xi)=\vp_0(\xi)+\e\vp_1(\xi) 
\end{equation*}
where $\vp_0$ is the unique solution of 
\begin{equation*}
\begin{cases}
  -\vp_0''(\xi)=g(\xi)\phi(\e\xi) \\
   \vp_0(\pm1)=K(\e)e^{-\a\e}.
\end{cases}
\end{equation*}
This yields the following explicit formula for $\vp_0$
\begin{equation}\label{varpi0}
 \vp_0(\xi)=K(\e)e^{-\a\e}-\int_{-1}^\xi\int_0^z g(u)\phi(\e u)dudz.
\end{equation}
Now we introduce the operator
\begin{equation*}
\ML : \begin{cases}
       X & \longrightarrow X \\
       \psi & \longmapsto \left\{ \xi\mapsto \int_{-1}^\xi\int_0^z\lp\e\a^2+\nu(u) \rp\psi(u)dudz \right\}   
      \end{cases}
 \end{equation*}
where $X=\{\psi\in C^1(-1,1),\psi \textrm{ is even}\}$ endowed with the $C^1$ norm.
$\ML$ is obviously a bounded operator and $\vp_1$ satisfies
\begin{equation*}
 \lp I-\e\ML\rp\vp_1=\ML\vp_0.
\end{equation*}
Hence there exists a constant $C,$ for $\e$ small enough, $\displaystyle \lV\vp_1\rV_{C^1(-1,1)}\leq C\lV\vp_0\rV_{C^1(-1,1)}.$
We also have the integral equation for $\vp_1$
\begin{equation*}
 \vp_1(\xi)=\int_{-1}^\xi\int_0^z\lp\e\a^2+\nu(u) \rp\lp \vp_0+\e\vp_1\rp(u) dudz.
\end{equation*}
The continuity of the derivative in 1 gives
\begin{equation}\label{continuitederivee}
 \vp_0'(1)+\e\vp_1'(1)=-\e\a K(\e)e^{-\a\e}.
\end{equation}
The computation done in the previous paragraph yields: 
\begin{align}
 \vp_0'(1) & = -\int_0^1 g(u)\phi(\e u)du \nonumber \\
 \vp_0'(1) & = -\e\frac{1}{2}I_0 + o(\e) \label{phi0prime1}
\end{align}
where $I_0$ is defined by (\ref{integralelimite}). 
Using the integral equation for $\vp_1,$ the previous domination, (\ref{varpi0}) and at last the convergence of $\phi$ as $\e$ goes to 0,
there exists a constant $M$ such that
\begin{align}
 \vp_0'(1) & = \int_0^1\lp\e\a^2+\nu(u) \rp\lp \vp_0+\e\vp_1\rp(u) du \nonumber \\
  & = \frac{\nub}{2}K(\e)-\int_0^1\nu(\xi)\int_{-1}^\xi\int_0^z g(u)\phi(\e u)dudzd\xi+O\lp\e.K(\e)\rp \nonumber \\
 \vp_0'(1) & = \frac{\nub}{2}K(\e)-M+O\lp\e\lp1+K(\e)\rp\rp. \label{phi1prime1}
\end{align}
Insert (\ref{phi0prime1}) and (\ref{phi1prime1}) in (\ref{continuitederivee}) and we get
\begin{equation}
 \lim_{\e\to0}\sup \ |K(\e)|<+\infty,
\end{equation}
which provides with (\ref{varpi0}) and (\ref{yplusgdqueeps}) the boundedness of $\lV\vp\rV_{L^\infty(\R)}$ as $\e$ goes to 0.
Moreover, the bound is locally uniform on $\la,c.$

\subparagraph*{Convergence of $\vp,$ conclusion of the proof} We return to the initial variable.
Let $K_0$ be any limit point of $\lp K(\e)\rp_\e.$ Then a subsequence of $\lp \vp\rp_\e$ converges in the sense of distributions to
\begin{equation*}
 \vp_l(y)=K_0 e^{-\sqrt{P(\la)}|y|}
\end{equation*}
and $\vp_l$ satisfies in the sense of distributions
\begin{equation*}
 -\vp_l''(y)+\lp P(\la)+\nub\d_0\rp\vp_l(y)=I_0\d_0
\end{equation*}
whose unique solution is 
\begin{equation}\label{limitevp}
 \vp_l:y\longmapsto \frac{I_0}{\nub+2\sqrt{P(\la)}}e^{-\sqrt{P(\la)}|y|}.
\end{equation}
Being the only possible limit point, (\ref{limitevp}) is the limit of $(\vp)$ as $\e$ goes to 0. $I_0$ is negative, and
so is $\vp_l.$ The uniform boundedness allows the derivation in (\ref{eqPsi2}), and the proof is concluded.
\qed

\subsection{Proof of Corollary \ref{corollaireVitesse}}
 Let $c^*_0$ the spreading speed associated to the limit model (\ref{BRReq2}),
$(c^*_0,\la^*_0,\phi_0(c^*_0,\la^*_0))$ the corresponding linear travelling wave.
We consider $\Psi_2$ as a function of $(\la,c;\e),$ $\Psi_1$ as a function of $(\la,c).$ We have 
$\Psi_1(\la^*_0,c^*_0)=\Psi_2(\la^*_0,c^*_0;0)$ and $\Psi_1(\la^-_2,c^*_0)<\Psi_2(\la^-_2,c^*_0;0).$ 
Hence ,there exists $\las,$ $\la_2^-<\las<\la_0^*,$ $\Psi_2(\las,c^*_0;0)=\Psi_1(\la_2^-,c^*_0).$
Let $V$ be any open set in $]\la_2^-,\la_2^+[$ containing $\las$ and $\la_0^*.$ From Proposition \ref{deriveePsi2positive}, there
exists $\e_0$ such that $\forall \e<\e_0,$ $\forall \la\in V,$ $\Psi_2(\la,c^*_0;\e)>\Psi_2(\la,c^*_0;0).$ From the definition of $\las,$
it yields $\Psi_2(\la,c^*_0;\e)>\Psi_1(\la,c^*_0),$ $\forall \la.$ The monotonicity of $\Psi_1$ and $\Psi_2$ with respect to $c$ concludes the proof.
\qed

\section{The semi-limit case (\ref{RPSL}): non optimality of concentrated exchanges}
Considering the above result, it may seem natural that in the case (\ref{RPSL}), that is
when exchanges from the road to the field are localised on the road, the spreading speed would be minimal
for localised exchange from the field to the road. The purpose of this section is
the proof of Theorem \ref{thmdenonacceleration} in which we show that any behavour may happen in the neighbourhood of a Dirac measure.
For the sake of convenience, throughout this section we set 
$$
d=\nub=1.
$$
Let us recall that we consider exchange terms $\nu$ of the form
\begin{equation}
 \nu(y)=(1-\e)\d_0+\e\up(y)
\end{equation}
where 
\begin{equation*}
 \up\in\Lambda_1:=\{\up\in C_0(\R),\up\geq 0, \int\up=1,\up \textrm{ is even} \}.
\end{equation*}

Let $c^*_0$ the spreading speed associated to the limit model (\ref{BRReq2}),
$(c^*_0,\la^*_0,\phi_0(c^*_0,\la^*_0))$ the corresponding linear travelling waves. The $(c,\la,\phi)-$equation associated to the
system (\ref{RPSL}) with exchange term of the form (\ref{nuperturbe}) is as follows, completed by evenness:
\begin{equation}\label{eqphinupert}
\begin{cases}
 -\phi''(y)+\lp-\fp+\la c-\la^2+\up(y)\rp\phi(y) = 0 \qquad y>0\\
 \phi'(0)=\frac{1}{2}\lp(1-\e)\phi(0)-\mub\rp, \qquad \phi\in H^1(\R^+).
 \end{cases}
\end{equation}
The associated function $\Psi_2$ is given by
\begin{equation}\label{Psi2pert}
 \Psi_2(\la,c)=(1-\e)\phi(0)+\e\int_\R \up\phi
\end{equation}
where $\phi$ is the unique solution of (\ref{eqphinupert}). What we have to show is that, in
a neighbourhood of $(\la^*_0,c^*_0),$ the difference $\lp\Psi_2^0(\la,c)-\Psi_2(\la,c)\rp$ is of constant sign for $\e$
small enough, and that this sign can be different depending on the parameters $D,\mub,\fp.$ Once again, for the sake of simplicity 
and as long as
there is no possible confusion, we set 
$$
P(\la)=-\fp+\la c-\la^2, \qquad \a=\sqrt{P(\la)}.
$$

Of course, we are looking for function $\phi$ of the form 
\begin{equation}\label{ansatz}
\phi = \phi_0+\e\phi_1 
\end{equation}
where $\phi_0$ is solution of (\ref{eqsurphiBRR}). Hence, $\phi_1$ satisfies
\begin{equation}\label{eqphi1}
 -\phi_1''+(\d_0+\a^2)\phi_1 = \lp\d_0-\up\rp\lp\phi_0+\e\phi_1\rp.
\end{equation}

\begin{lem}\label{bornephi1}
 Let $\a_0>0.$ There exist $\e_0>0,$ $K>0,$ depending only on $\a_0,$ such that $\forall \e<\e_0,$
 $$
 \lV\phi_1\rV_{L^\infty} \leq K \lV\phi_0\rV_{L^\infty}
 $$
 where $\phi_1$ is the solution of (\ref{eqphi1}). We may also keep in mind that 
 $\lV\phi_0\rV_{L^\infty}=\phi_0(0).$ We can see in (\ref{solPsi2BRR}) that it is
 uniformly bounded in $\a,D,\fp.$
\end{lem}

\begin{proof}
We introduce the operator
\begin{equation*}
\ML : \begin{cases}
       X & \longrightarrow X \\
       \psi & \longmapsto \vp       
      \end{cases}
 \end{equation*}
where $X=\{\psi\in BUC(\R),\psi \textrm{ is even}\}$ and $\vp$ is the only bounded solution of 
\begin{equation}
 -\vp''+(\a^2+\d_0)\vp = \lp\d_0-\up\rp\psi.
\end{equation}
From (\ref{eqphi1}), it is easy to see that $\phi_1$ satisfies
$
\phi_1=\ML\phi_0+\e\ML\phi_1.
$
As $\up$ and $\phi_0$ are even, we focus on $\ML$ defined for bounded, uniformly continuous even functions.
Let $\psi\in BUC(\R)$ be any even function, and $\vp:=\ML\psi.$ That is, $\vp$ satisfies 
\begin{equation}
 \begin{cases}
 -\vp''+\a^2\vp = -\up\psi \qquad y>0 \\
 \vp'(0)=\frac{1}{2}\lp\vp(0)-\psi(0)\rp.
 \end{cases}
\end{equation}
As in the previous section, a simple computation gives
\begin{align}
 \vp(y) = & -\frac{e^{\a y}}{2\a}\int_y^\infty e^{-\a z}(\up\psi)(z)dz \label{calculL} \\
  & + e^{-\a y}\lp \frac{\psi(0)}{1+2\a}+\frac{1-2\a}{2\a(1+2\a)}\int_0^\infty e^{-\a z}(\up\psi)(z)dz -\frac{1}{2\a}\int_0^y e^{\a z}(\up\psi)(z)dz \rp. \nonumber
\end{align}
Recall that $\up$ is nonnegative and of weight 1, and $\a=\sqrt{P(\la)}>0$. A rough majoration in (\ref{calculL}) yields
\begin{equation}\label{normeoperator}
 \lV\vp\rV_{L^\infty} \leq  \lV\psi\rV_{L^\infty}\lp\frac{1}{1+2\a}+\frac{\lb 1-2\a\rb}{4\a(1+2\a)}+\frac{1}{4\a}\rp.
\end{equation}
Hence $\ML$ is a bounded linear operator, with norm $\lV\ML\rV$ depending on $\a,$ and uniformly bounded on $\a>\a_0>0.$
For $\e$ small enough, $(I-\e\ML)$ is invertible with bounded inverse and 
\begin{equation}\label{operator}
 \phi_1 = \lp I-\e\ML\rp^{-1}\ML \phi_0.
\end{equation}
Moreover, $\phi_1$ satisfies the integral equation $\phi_1=\ML(\phi_0+\e\phi_1)$ given by (\ref{calculL}).
Combining (\ref{operator}) with (\ref{normeoperator}) concludes the proof of Lemma \ref{bornephi1}.
\end{proof}

\subparagraph{The difference $\Psi_2^0-\Psi_2$} The function $\Psi_2$ is given by (\ref{Psi2pert}) with $\phi$ of the form (\ref{ansatz}).
Then, using Lemma \ref{bornephi1}, for all $\a>\a_0,$
\begin{align}
 \Psi_2^0-\Psi_2  & = \phi_0(0)-(1-\e)\lp\phi_0(0)+\e\phi_1(0)\rp-\e\int_\R\lp\phi_0+\e\phi_1\rp\up \nonumber \\
  & = \e\lp\phi_0(0)-\phi_1(0)-\int_\R\up\phi_0\rp+o(\e). \label{difference1}
\end{align}
It appears necessary to compute $\phi_1(0).$ Equation (\ref{calculL}) gives
\begin{align}
 \phi_1(0) = & \lp\frac{1-2\a}{2\a(1+2\a)}-\frac{1}{2\a}\rp\int_\R^\infty e^{-\a y}\up(y)\lp\phi_0(y)+\e\phi_1(y)\rp+
 \frac{1}{1+2\a}\lp\phi_0(0)+\e\phi_1(0)\rp \nonumber \\
  = & \frac{\phi_0(0)}{1+2\a}-\frac{2}{1+2\a}\int_0^\infty e^{-\a y}\up(y)\phi_0(y)dy + O(\e) \nonumber \\
  = & \frac{\phi_0(0)}{1+2\a}\lp 1- \int_\R e^{-2\a|y|}\up(y)dy\rp +O(\e). \label{phi1zero}
\end{align}
Now recall that $\up$ is of mass 1 and, using (\ref{phi1zero}) in (\ref{difference1}),
\begin{align}
 \Psi_2^0-\Psi_2 = & \e\phi_0(0)\int_\R\up(y)\lp 1-e^{-\a|y|}-\frac{1}{1+2\a}\lp 1-e^{-2\a|y|}\rp\rp dy+o(\e) \nonumber \\
  = & \e\phi_0(0)\int_\R\up(y)g(\a,y) dy+o(\e). \label{difference2}
\end{align}
The function $g$ is obviously even in $y$, and smooth on ${\R_*^+}^2.$ We can easily see that:
\begin{itemize}
 \item if $\a\geq\frac{1}{2},$ then $\forall y>0,\ g(\a,y)>0.$
 \item If $\a<\frac{1}{2},$ then there exists $y(\a)$ such that, in a neighbourhood of $\a,$ $\forall|y|<y(\a),$ $g(.,y)<0.$ 
\end{itemize}
We are interested in the local behaviour near $(\la^*,c_0^*).$ Hence, $g(\a,y)$ has to be considered near 
$\a^*:=\sqrt{-\fp+c^*_0\la^*-{\la^*}^2}=P(\la^*).$ 
\subparagraph{Perturbation enhancing the velocity: $\a^*<1/2$}
$P$ achieves its maximum at $\la=\frac{c}{2}$ and $c\mapsto P(\frac{c}{2})$ is nondecreasing. From 
\cite{BRR1} we know that $c^*_0$ satisfies
$$
\frac{c^*_0}{D}\leq \frac{c^*_0-\sqrt{{c^*_0}^2-c_K^2}}{2}
$$
where $c_K=2\sqrt{\fp}$ is the classical spreading speed for KPP-type reaction-diffusion. It follows easily that
$c^*_0\leq \frac{D\sqrt{\fp}}{\sqrt{D-1}}$ which, combined with the two upper remarks, yields the following sufficient condition
for $\a^*=\sqrt{P(\la^*)}$ to be less than $1/2$:
\begin{equation}\label{conditionm1}
 D < 2+\frac{1}{2\fp}+\frac{1}{2}\sqrt{12+(\frac{1}{\fp})^2+\frac{7}{\fp}}:=m_1.
\end{equation}
Hence, provided the condition (\ref{conditionm1}) holds, $\a^*<1/2$ and there exists $y(\a^*)$ and a neighbourhood $\MV$ of $\a^*$
such that $g(\a,y)<0$ for $|y|<y(\a^*)$ and $\a\in\MV.$ Take $\up$ such that $\textrm{supp}(\up)\subset]-y(\a^*),y(\a^*)[$
and, for all $\a\in\MV,$ that is for all $\la$ in a neighbourhood of $\la^*,$ for $\e$ small enough, 
$$
\lp\Psi_2^0-\Psi_2\rp(\la,c^*_0)<0.
$$
The result follows from the monotonicity of $\Psi_2$ with respect to $c.$

\subparagraph{Locally maximal velocity for for $\nu=\d_0:$ $\a^*>1/2$: proof of Theorem \ref{thmdenonacceleration}}
It remains to show that $\a^*$ can be greater than $\frac{1}{2}.$ We will need the second part of Proposition \ref{asymptoticDfp}.
  From now, we fix an exchange rate $\mub>4.$ We will use the fact that, at $(c^*_0,\la^*),$ 
 \begin{equation}\label{egalitederivees}
\frac{d}{d\la}\lp\Psi_1-\Psi_2^0\rp(\la)=0.
 \end{equation}
 Explicit computation gives 
 \begin{equation}\label{deriveesPsi}
\begin{cases}
\frac{d}{d\la}\Psi_1(\la) & = -2D\la+c \\
\frac{d}{d\la}\Psi_2^0(\la) & = -\frac{\mub(c-2\la)}{\sqrt{P(\la)}\lp 1+2\sqrt{P(\la)}\rp^2}.
\end{cases}
 \end{equation}
 Recall that $\la^*$ has to satisfy
 $$
 \frac{c^*_0}{D}\leq \la^*\leq \la_1^+:=\frac{c^*_0+\sqrt{{c^*_0}^2+4D\mub}}{2D}.
 $$
 Now applying Lemma \ref{asymptoticDfp}, for all $\d>0$ there exists $M>0,$
 $\fp,D>M$ entails 
 $\lb\Psi_1'(\la^*)-c^*_0\rb<\d$ and 
 $\lb \la_2^-\rb<\d$ (recall that $\la_2^-=\frac{c-\sqrt{c^2-c_K^2}}{2}$).
 To prove that $\a^*=\sqrt{P(\la^*)}>1/2,$ we distinguish two cases.
 
 First case: $\displaystyle\la^*>\frac{1}{2}\lp \la_2^-+\frac{c^*_0}{2}\rp.$ Thus
 $\displaystyle\la^*>\frac{c^*_0}{4}-\d$ which yields with Lemma \ref{asymptoticDfp}
 $$
 P(\la^*)=D\fp\frac{3}{16}-\fp+O(\d D\fp)>\frac{1}{4}.
 $$
 
 Second case: $\displaystyle\la_2^-<\la^*<\frac{1}{2}\lp \la_2^-+\frac{c^*_0}{2}\rp.$
Thus, from (\ref{egalitederivees}), (\ref{deriveesPsi}) and the above inequalities
given by Lemma \ref{asymptoticDfp},
$$
c^*_0+\d>\lb{\Psi_2^0}'(\la^*)\rb=\lb\frac{\mub(c^*_0-2\la^*)}{\a^*(1+2\a^*)^2}\rb
>\frac{2(c^*_0-2\d)}{\a^*(1+2\a^*)}
$$
which implies $\a*>1/2$ for $\d$ small enough. This concludes the proof of
Theorem \ref{thmdenonacceleration}.
\qed

\newpage
\bibliographystyle{plain}
\footnotesize
\bibliography{biblio}

\begin{thebibliography}{10}

\bibitem{AW}
D.~G. Aronson and H.~F. Weinberger.
\newblock Multidimensional nonlinear diffusion arising in population genetics.
\newblock {\em Adv. Math.}, 30:33--76, 1978.

\bibitem{Berestycki1}
H.~Berestycki.
\newblock Le nombre de solutions de certains probl\`emes semi-lin\'eaires
  elliptiques.
\newblock {\em J. Funct. Anal.}, 40(1):1--29, 1981.

\bibitem{BRRC}
H.~Berestycki, A.-C. Coulon, J.-M. Roquejoffre, and L.~Rossi.
\newblock Exponentially fast fisher-{KPP} propagation induced by a line of
  integral diffusion.
\newblock {\em forthcoming}, 2014.

\bibitem{BHN1}
H.~Berestycki, F.~Hamel, and N.~Nadirashvili.
\newblock The speed of propagation for {KPP} type problems. {I}. {P}eriodic
  framework.
\newblock {\em J. Eur. Math. Soc. (JEMS)}, 7(2):173--213, 2005.

\bibitem{BHN2}
H.~Berestycki, F.~Hamel, and N.~Nadirashvili.
\newblock The speed of propagation for {KPP} type problems. {II}. {G}eneral
  domains.
\newblock {\em J. Amer. Math. Soc.}, 23(1):1--34, 2010.

\bibitem{BHR}
H.~Berestycki, F.~Hamel, and L.~Roques.
\newblock Analysis of the periodically fragmented environment model : I -
  species persistence.
\newblock {\em Journal of Mathematical Biology}, 51(1):75--113, 2005.

\bibitem{BRR2}
H.~Berestycki, J.-M. Roquejoffre, and L.~Rossi.
\newblock Fisher-{KPP} propagation in the presence of a line: further effects.
\newblock {\em Nonlinearity}, 26(9):2623--2640, 2013.

\bibitem{BRR1}
H.~Berestycki, J.-M. Roquejoffre, and L.~Rossi.
\newblock The influence of a line with fast diffusion on {F}isher-{KPP}
  propagation.
\newblock {\em Journal of Mathematical Biology}, 66(4-5):743--766, 2013.

\bibitem{Cartan}
H.~Cartan.
\newblock {\em Calcul diff\'erentiel}.
\newblock Hermann, Paris, 1967.

\bibitem{these_AC}
A.-C. Coulon-Chalmin.
\newblock {\em Propagation in reaction-diffusion equations with fractional
  diffusion}.
\newblock PhD thesis, Toulouse 3 - UPC Barcelona, 2014.

\bibitem{Dietrich2}
L.~Dietrich.
\newblock Velocity enhancement of reaction-diffusion fronts by a line of fast
  diffusion.
\newblock {\em preprint}, 2014.

\bibitem{Dietrich1}
L.~Dietrich.
\newblock Existence of fronts in a reaction-diffusion system with a line of
  fast diffusion.
\newblock {\em AMRX}, to appear, 2015.

\bibitem{fisher}
R.~A. Fisher.
\newblock the advance of advantageous genes.
\newblock {\em ANN. Eugenics}, 7:335--369, 1937.

\bibitem{FG}
Ju. Gertner and M.~I. Fre{\u\i}dlin.
\newblock The propagation of concentration waves in periodic and random media.
\newblock {\em Dokl. Akad. Nauk SSSR}, 249(3):521--525, 1979.

\bibitem{GT}
D.~Gilbarg and N.~S. Trudinger.
\newblock {\em Elliptic Partial Differential Equations of Second Order}.
\newblock Springer-Verlag, 2001.

\bibitem{Glangetas_92}
L.~Glangetas.
\newblock \'{E}tude d'une limite singuli\`ere d'un mod\`ele intervenant en
  combustion.
\newblock {\em Asymptotic Anal.}, 5(4):317--342, 1992.

\bibitem{henry}
D.~Henry.
\newblock {\em Geometric Theory of Semilinear Parabolic Equation}.
\newblock Lecture Notes in Mathematics. Springer-Verlag, 1981.

\bibitem{KPP}
A.~Kolmogorov, I.~Petrovsky, and N.~Piskounov.
\newblock Etude de l'\'equation de la diffusion avec croissance de la
  quantit\'e de mati\`ere et son application \`a un probl\`eme biologique.
\newblock {\em Bull. Univ. Etat Moscou}, 1:1--26, 1937.

\bibitem{matano}
X.~Liang, X.~Lin, and H.~Matano.
\newblock A variational problem associated with the minimal speed of travelling
  waves for spatially periodic reaction-diffusion equations.
\newblock {\em Trans. Amer. Math. Soc.}, 362(11):5605--5633, 2010.

\bibitem{Pauthier2}
A.~Pauthier.
\newblock Uniform dynamics for fisher-{KPP} propagation driven by a line of
  fast diffusion under a singular limit.
\newblock {\em preprint}, 2014.

\bibitem{W2002}
H.F. Weinberger.
\newblock On spreading speeds and traveling waves for growth and migration
  models in a periodic habitat.
\newblock {\em J. Math. Biol.}, 45(6):511--548, 2002.

\end{thebibliography}
\end{document}